\documentclass[12pt]{amsart}
\usepackage{}

\usepackage{amsmath}
\usepackage{amsfonts}
\usepackage{amssymb}
\usepackage[all]{xy}           

\usepackage{bbding}
\usepackage{txfonts}
\usepackage{amscd}

\usepackage[shortlabels]{enumitem}
\usepackage{ifpdf}
\ifpdf
  \usepackage[colorlinks,final,backref=page,hyperindex]{hyperref}
\else
  \usepackage[colorlinks,final,backref=page,hyperindex]{hyperref}
\fi
\usepackage{tikz}
\usepackage[active]{srcltx}

\makeatletter

\newtheorem{thm}{Theorem}[section]
\newtheorem{lem}[thm]{Lemma}
\newtheorem{cor}[thm]{Corollary}
\newtheorem{pro}[thm]{Proposition}
\newtheorem{ex}[thm]{Example}
\newtheorem{rmk}[thm]{Remark}
\newtheorem{defi}[thm]{Definition}

\newcommand {\emptycomment}[1]{}

\newcommand{\be }{\begin{equation}}
\newcommand{\ee }{\end{equation}}

\newcommand{\Real}{\mathbb R}
\newcommand{\Comp}{\mathbb C}

\newcommand{\huaL}{\mathcal{L}}

\newcommand{\huaE}{\mathcal{E}}
\newcommand{\huaF}{\mathcal{F}}

\newcommand{\huaV}{\mathcal{V}}

\newcommand{\huaP}{\mathcal{P}}
\newcommand{\huaD}{\mathcal{D}}

\newcommand{\huaK}{\mathcal{K}}

\newcommand{\huaM}{\mathcal{M}}
\newcommand{\CWM}{C^{\infty}(M)}

\newcommand{\frkg}{\mathfrak g}

\newcommand{\frkE}{\mathfrak E}
\newcommand{\frkF}{\mathfrak F}

\newcommand{\frkV}{\mathfrak V}

\newcommand{\frkX}{\mathfrak X}

\newcommand{\half}{\frac{1}{2}}

\newcommand{\Courant}[1]{\left\llbracket  #1\right\rrbracket }

\newcommand{\br}[1]{   [ \cdot,    \cdot  ]   }

\newcommand{\dev}{\mathfrak{D}}

\newcommand{\id}{\rm{id}}
\newcommand{\Id}{\rm{Id}}

\newcommand{\g}{\mathfrak g}

\newcommand{\dM}{\mathrm{d}}

\newcommand{\Hom}{\mathrm{Hom}}
\newcommand{\Der}{\mathrm{Der}}

\newcommand{\PRin}{\mathrm{PRin}}
\newcommand{\Rin}{\mathrm{Rin}}
\newcommand{\PLie}{\mathrm{PLie}}
\newcommand{\Lie}{\mathrm{Lie}}
\newcommand{\crmod}{\mathrm{{\bf crmod}}}

\newcommand{\gl}{\mathfrak {gl}}
\newcommand{\sln}{\mathfrak {sl}}

\newcommand{\Ker}{\mathrm{ker}}

\newcommand{\Rank}{\mathrm{Rank}}
\newcommand{\coker}{\mathrm{coker}}

\newcommand{\img}{\mathrm{im}}

\newcommand{\K}{\mathbb{K}}

\begin{document}

\title[pre-Lie Rinehart algebras ]{Cohomologies and crossed modules for pre-Lie Rinehart algebras}

\author{Liangyun Chen}
\address{School of Mathematics and Statistics, Northeast Normal University,\\
 Changchun 130024, Jilin, China }
\email{chenly640@nenu.edu.cn}

\author{Meijun Liu}
\address{School of Mathematics and Statistics, Northeast Normal University,\\
 Changchun 130024, Jilin, China }
\email{liumj281@nenu.edu.cn}

\author{Jiefeng Liu$^{\ast}$}
\address{School of Mathematics and Statistics, Northeast Normal University,\\
 Changchun 130024, Jilin, China }
\email{liujf534@nenu.edu.cn}
\thanks{$^{\ast}$ the corresponding author}
\vspace{-5mm}


\begin{abstract}
A pre-Lie-Rinehart algebra is an algebraic generalization of the notion of a left-symmetric algebroid. We construct pre-Lie-Rinehart algebras from $r$-matrices through Lie algebra actions. We study  cohomologies of pre-Lie-Rinehart algebras and show that abelian extensions of pre-Lie-Rinehart algebras  are classified by the second cohomology groups. We introduce the notion of crossed modules for pre-Lie-Rinehart algebras and show that they are classified by the third cohomology groups of pre-Lie-Rinehart algebras. At last, we use (pre-)Lie-Rinehart $2$-algebras to characterize the crossed modules for (pre-)Lie Rinehart algebras.
\end{abstract}





\maketitle

\tableofcontents

\allowdisplaybreaks


\section{Introduction}\label{sec:intr}
The notion of a Lie-Rinehart algebra is an algebraic generalization of Lie algebroids. It was first introduced by Rinehart in his seminal paper \cite{Rine}, in which he used this notion to develop a formalism of differential forms for general commutative algebras. In \cite{Hueb1}, Huebschmann  studied Lie-Rinehart algebras systematically and emphasized their important applications in Poisson geometry. See \cite{CaLaPi2,Dokas,Hueb2,Huebs3,Huebs4} for more details and applications on  Lie-Rinehart algebras.

\emptycomment{Lie-Rinehart algebras were first introduced by Rinehart who defined cohomology groups for the category of Lie-Rinehart algebras with coefficients in a Lie-Rinehart module in 1963 in \cite{Rine}. J. Huebschmann  has been studied further developments in 1990 in \cite{Hueb1} and  emphasized their important applications in Poisson geometry \cite{{Huebs1},{Huebs2}}.  J. M. Casas, M. Ladra and T. Pirashvili  studied  crossed modules and triple cohomology for Lie-Rinehart algebras in \cite{{CFGK1},{CL}}. Loday and Vallette remark that a Lie-Rinehart algebra is a Poisson algebra in \cite{LV}. Lie-Rinehart structures have been the subject of extensive studies, relations to symplectic geometry, poisson structures and etc. See more details for  Lie-Rinehart algebras in \cite{{CLZ},{Dokas},{Vitagliano}}.}

Pre-Lie algebras (or left-symmetric algebras) arose from the study of convex homogeneous cones (\cite{Vinberg}), affine manifolds and affine structures on Lie groups (\cite{Koszul}), deformation and cohomology theory of associative algebras (\cite{G}) and then appear in many fields in mathematics and mathematical physics. See the survey article \cite{Pre-lie algebra in geometry} and the references therein.  The notion of a pre-Lie-Rinehart algebra was introduced by Fl${\o}$ystad, Manchon and Munthe-Kaas in \cite{FMM}. They organized colored aromatic trees into a pre-Lie-Rinehart algebra endowed with a natural trace map, which yields
the algebraic foundations of aromatic B-series. The importance of a pre-Lie-Rinehart algebra is that the commutator gives rise to a Lie-Rinehart algebra and the left multiplication gives rise to a representation of the commutator Lie-Rinehart algebra.  A pre-Lie-Rinehart algebra is also an algebraic generalization of the notion of a left-symmetric algebroid (\cite{LiuShengBaiChen}). See \cite{BBo,lsb,lsb2,WLS} for more details on left-symmetric algebroids. Recently, the authors in \cite{BCEM} studies the cohomology and deformation of the pre-Lie-Rinehart algebra (they called it a left-symmetric Rinehart algebra).

The notion of crossed modules was introduced by Whitehead in \cite{White} to study relative homotopy groups. After that, crossed modules have been one of the fundamental concepts in homotopy theory, algebraic $K$-theory, combinatorial group theory and homological algebras. Crossed modules for Lie algebras were studied in \cite{G2,Kassel} and it was shown that they are classified by the third cohomology groups of Lie algebras. Lie algebras can be categorified to Lie $2$-algebras. It is well-known that the category of strict Lie $2$-algebras is isomorphic to the category of crossed modules for Lie algebras (\cite{BC}). The author in \cite{Sheng19} introduced the notion of crossed modules for pre-Lie algebras and gave the equivalence between the category of strict Lie $2$-algebras and the category of crossed modules for pre-Lie $2$-algebras, in which the pre-Lie $2$-algebra is a categorification of the pre-Lie algebra.  As a natural generalization of crossed modules for Lie algebras, the authors in \cite{CaLaPi1} introduced the notion of crossed modules for Lie-Rinehart algebras and proved that they are classified by the third cohomology groups of Lie-Rinehart algebras.

In this paper, we introduce the notion of crossed modules for pre-Lie-Rinehart algebras and classify crossed modules via the third cohomology groups of pre-Lie-Rinehart algebras. Then we give the notion of (pre-)Lie-Rinehart $2$-algebras, which is the algebraic generalization of (pre-)Lie $2$-algebroids given in \cite{LS,SZ17}. Furthermore, we establish a one-to-one correspondence between strict (pre-)Lie-Rinehart $2$-algebras and crossed modules for (pre-)Lie Rinehart algebras.

\emptycomment{In this paper, we recall the notion of pre-Lie-Rinehart algebras and construct some examples of pre-Lie-Rinehart algebras in different ways. Particularly, we construct pre-Lie-Rinehart algebras by $r$-matrices over $\sln(2,\K)$ in Example \ref{ex:$r$-matrices}. We study the cohomology and extensions of pre-Lie-Rinehart algebras. According to the the classical argument in Eilenberg-Mac Lane cohomology, we bulid  a one-to-one corresponding between the equivalent classes of
extensions and the second cohomology classes in Theorem \ref{thm:secind equivalent classes}.

Our main result is relate the crossed module with the cohomology of pre-Lie-Rinehart algebras. We show that crossed module of pre-Lie-Rinehart algebras can be used to define crossed extensions. Then, we establish a bijection between the third cohomology of pre-Lie-Rinehart algebras and the set of equivalence classes of crossed extensions in Theorem \ref{crossed module of pLA}. Finally, we introduce the definition of strict (pre-)Lie-Rinehart $2$-algebras and establish a one-to-one correspondence between strict (pre-)Lie-Rinehart $2$-algebras and crossed modules of (pre-)Lie Rinehart algebras in Theorem \ref{one to one correspondence}.}

The paper is organized as follows.
In Section \ref{sec:MC-RRB-operator}, first we review the representations and cohomologies of  Lie-Rinehart algebras. Then we recall the notion of a pre-Lie-Rinehart algebra and  illustrate it by some examples. In Section \ref{sec:$r$-matrices}, we use $r$-matrices to construct pre-Lie-Rinehart algebras. In Section \ref{sec:Cohomology and extensions}, we study the cohomologies of pre-Lie-Rinehart algebras and show that abelian extensions of pre-Lie-Rinehart algebras  can be classified by the second cohomology groups. In Section \ref{sec:crossed module}, we introduce the notion of crossed modules for pre-Lie-Rinehart algebras and show that they are classified by the third cohomology groups of pre-Lie-Rinehart algebras. In section \ref{sec:Lie-Rinehart $2$-algebras}, we introduce the notion of (pre-)Lie-Rinehart $2$-algebras and then use strict (pre-)Lie-Rinehart $2$-algebras to characterize the crossed modules for (pre-)Lie Rinehart algebras.

\section{Lie-Rinehart algebras and pre-Lie-Rinehart algebras}\label{sec:MC-RRB-operator}
In this section, we first recall the representations and cohomologies of Lie-Rinehart algebras. Then we introduce the notion of a pre-Lie-Rinehart algebra and give some constructions of pre-Lie-Rinehart algebras.
\subsection{Lie-Rinehart algebras}
Let $\K$ be a field and $A$ be a unitary commutative algebra over $\K$. A derivation of $A$ is a $\K$-linear map $D:A\longrightarrow A$ satisfying the Leibniz rule $D(ab)=aD(b)+D(a)b$ for all $a,b\in A$. The $A$-module $\Der(A)$ of all derivations of $A$ is a Lie $\K$-algebra under the commutator $[D_1,D_2]^c:=D_1\circ D_2-D_2\circ D_1$ for $D_1,D_2\in\Der(A)$.
\begin{defi}
A {\bf Lie-Rinehart algebra} consists of the following data:
\begin{itemize}
\item[$\bullet$]a unitary commutative $\K$-algebra $A$;
\item[$\bullet$]an $A$-module $E$;
\item[$\bullet$]an $A$-module map $\theta_E:E\longrightarrow\Der(A)$, called the anchor;
\item[$\bullet$]a $\K$-linear Lie bracket $[-,-]_E:E\otimes E\to E$.
\end{itemize}
These data satisfy the following additional conditions:
\begin{eqnarray}
 [X,aY]_E&=&a[X,Y]_E+\theta_E(X)(a)Y,\\
 \theta_E([X,Y]_E)&=&[\theta_E(X),\theta_E(Y)]^c,\quad \forall~X,Y\in E,a\in
A.
\end{eqnarray}
We denote a Lie-Rinehart algebra by $(E,[-,-]_E,\theta_E)$ if $A$ is fixed.
\end{defi}

\emptycomment{\begin{rmk}
 If the $A$-module $E$ is faithful, the condition that $\theta_E:E\longrightarrow\Der(A)$ is a Lie algebra morphism in the definition of Lie-Rinehart algebra is not necessary. This follows from
 $$[X,[Y,aZ]_E]_E+c.p.=a\big([X,[Y,Z]_E]_E+c.p.\big)+\big([\theta_E(X),\theta_E(Y)]^c-\theta_E([X,Y]_E)\big)(a)Z,$$
 where $X,Y,Z\in E$ and $a\in A$.
\end{rmk}}

\begin{ex}
Let $(\huaL,[-,-]_\huaL,\theta_\huaL)$ be a Lie algebroid on a manifold $M$. Then $(E=\Gamma(\huaL),[-,-]_\huaL,\theta_\huaL)$ is a Lie-Rinehart algebra with $\K=\Real$ and $A=\CWM$ the algebra of smooth functions on a manifold $M$.
\end{ex}

Recall that an {\bf action} of a Lie $\K$-algebra $(\g,[-,-]_\g)$ on a unitary commutative $\K$-algebra $A$ is a Lie algebra homomorphism for $\g$ to $\Der(A)$.

\begin{ex}
 Let $(\g,[-,-]_\g)$ be a Lie $\K$-algebra and $\lambda:\g\longrightarrow\Der(A)$ an action of $\g$ on $A$. Then the {\bf transformation Lie-Rinehart algebra} is $E=A\otimes \g$ with Lie bracket
 $$[a\otimes x,b\otimes y ]_\rho=ab\otimes [x,y]_\g+a\rho(x)(b)\otimes y-b\rho(y)(a)\otimes x$$
 and the anchor $\theta_\rho(a\otimes x)(b)=a\rho(x)(b)$, where $a,b\in A,x,y\in\g.$
\end{ex}

\emptycomment{A $\K$-linear map $\sigma:A\longrightarrow \huaE$ is called a derivation if it satisfies the Leibniz rule:
$$\sigma(ab)=\sigma(a)b+a\sigma(b),\forall~a,b\in A.$$
All this $\huaE$-value derivations form an $A$-module denoted by $\Der(A,\huaE)$. }
Let $\huaE$ be an $A$-module. A {\bf derivation} of $\huaE$ is a pair $(D,\sigma_D)$, where $D:\huaE\longrightarrow\huaE$ is a $\K$-linear map and $\sigma_D\in\Der(A)$ satisfying the following compatibility condition:
  $$D(a u)=a D(u)+\sigma_D(a)u,\quad \forall a\in A,u\in\huaE.$$
  Derivations of $\huaE$ form an $A$-module denoted by $\dev(\huaE)$.

\begin{ex}
  Let $\huaE$ be an $A$-module. Then $\dev(\huaE)$ is a Lie-Rinehart algebra with respect to the commutator bracket and the anchor $\theta_\dev$ is given by $\theta_\dev(D,\sigma)=\sigma$, which we call a {\bf gauge Lie-Rinehart algebra}.
\end{ex}
\begin{defi}
Let $(E,[\cdot,\cdot]_E,\theta_E)$ and $(F,[\cdot,\cdot]_F,\theta_F)$ be two Lie-Rinehart algebras over $A$. A  {\bf homomorphism}
from $E$ to $F$ is an $A$-module morphism $\varphi:E\longrightarrow F$ such
that
\begin{eqnarray*}
 \varphi[X,Y]_E=[\varphi(X),\varphi(Y)]_F,\quad \theta_F\circ\varphi=\theta_E,\quad\forall X,Y\in E.
\end{eqnarray*}
\end{defi}
\begin{defi}
A {\bf representation} of a Lie-Rinehart algebra $(E,[\cdot,\cdot]_E,\theta_E)$ on an $A$-module $\huaE$ is a homomorphism $\rho$ from $E$ to the gauge Lie-Rinehart algebra $\dev(\huaE)$.
\end{defi}
An $A$-module $\huaE$ equipped with a representation of $E$ is said to be an $E$-module.

For every Lie-Rinehart algebra $(E,[\cdot,\cdot]_E,\theta_E)$, $A$ is a $E$-module with $\rho=\theta_E$.

\emptycomment{Let $(E,[\cdot,\cdot]_E,\theta_E)$ be a Lie-Rinehart algebra and let $\huaK=\Ker(\theta_E)$. Then $\huaK$ becomes a $E$-module with
$$\rho(X)(Y)=[X,Y]_E,\quad\forall~X\in E,Y\in \huaK.$$}
\emptycomment{
Choose a representation  $(\huaE;\rho).$
The {\bf dual representation} of a Lie-Rinehart algebra $E$ on $\huaE^*:=\Hom_A(\huaE,A)$ is the $A$-module map $\rho^*:E\longrightarrow \Der(\huaE^*)$ given by
$$
\langle \rho^*(X)(\xi),u\rangle=\theta_E(X)\langle \xi,u\rangle-\langle \xi,\rho(X)(u)\rangle,\quad \forall~X\in E,~\xi\in \huaE^*,u\in\huaE.
$$}
Consider the $A$-module $E$ with representation $(\huaE;\rho)$, one defines for each $k\geq0$ the module of $k$-cochains on $E$ with coefficients in $A$ to be
$$\Omega^k(E,\huaE):=\Hom_A(\wedge^k E,\huaE).$$
 The corresponding coboundary operator
${\dM}:\Omega^k(E,\huaE)\longrightarrow
\Omega^{k+1}(E,\huaE)$ is given by
\begin{eqnarray*}
  {\dM}\varpi(X_1,\cdots,X_{k+1})&=&\sum_{i=1}^{k+1}(-1)^{i+1}\rho(X_i)\varpi(X_1\cdots,\widehat{X_i},\cdots,X_{k+1})\\
  &&+\sum_{1\leq i<j\leq k+1}(-1)^{i+j}\varpi([X_i,X_j]_E,X_1\cdots,\widehat{X_i},\cdots,\widehat{X_j},\cdots,X_{k+1}).
\end{eqnarray*}
We denote the $n$-th cohomology group of $(\bigoplus_{n\geq0}\Omega^{n}(E,\huaE),\dM)$ by $H_{\Rin}^n(E,\huaE)$.
\emptycomment{In particular, consider the $A$-module $E$ with representation $(A;\rho=\theta_E)$, an element $\varpi\in\Omega^2(E,A)$ is a {\bf 2-cocycle} if ${\dM} \varpi=0$, i.e.
\begin{equation}
  \theta_E(X)\varpi(Y,Z)- \theta_E(Y)\varpi(X,Z)+ \theta_E(Z)\varpi(X,Y)-\varpi([X,Y]_E,Z)+\varpi([X,Z]_E,Y)-\varpi([Y,Z]_E,X)=0.
\end{equation}

For $\varpi\in\wedge^2 E^*$, define $\varpi^\flat:E\longrightarrow \Hom_A(E,A)$ by
$$\varpi^\flat(X)(Y)=\varpi(X,Y),\quad X,Y\in E.$$
Define $\Ker~\varpi=\{X\in E\mid \varpi(X,Y)=0,\forall~Y\in E\}$. The element $\varpi\in\Omega^2(E,A)$ is called {\bf nondegenerate} if $\Ker~ \varpi=0$. This condition implies that $\varpi^\flat$ is injective and $\varpi^\flat$ may be not surjective. A {\bf symplectic Lie-Rinehart algebra} is a Lie-Rinehart algebra together with
a $2$-cocyle $\varpi\in\Omega^2(E,A)$ such that $\varpi^\flat$ is an isomorphism of $A$-modules. A sub-Lie-Rinehart algebra of a symplectic Lie-Rinehart algebra $(E,[-,-]_E,\theta_E,\varpi)$ is called {\bf Lagrangian} if it is maximal isotropic with respect to the skew-symmetric
bilinear form $\varpi$.}

\subsection{Pre-Lie-Rinehart algebras}\label{sec: pre-Lie-Rinehart algebras}

\begin{defi}
A {\bf pre-Lie $\K$-algebra} is a pair $(\frkg,\cdot_\frkg)$, where $\g$ is a vector space over $\K$, and  $\cdot_\frkg:\g\otimes \g\longrightarrow\g$ is a $\K$-bilinear operation
such that for all $x,y,z\in \g$, the associator
\begin{equation}\label{eq:associator}
(x,y,z)\triangleq x\cdot_\frkg(y\cdot_\frkg z)-(x\cdot_\frkg y)\cdot_\frkg z
\end{equation} is symmetric in $x,y$,
i.e.
$$(x,y,z)=(y,x,z),\;\;{\rm or}\;\;{\rm
equivalently,}\;\;x\cdot_\frkg(y\cdot_\frkg z)-(x\cdot_\frkg y)\cdot_\frkg z=y\cdot_\frkg(x\cdot_\frkg z)-(y\cdot_\frkg x)\cdot_\frkg
z.$$
\end{defi}

Let $(\frkg,\cdot_\frkg)$ be a pre-Lie $\K$-algebra. The commutator $
[x,y]^c=x\cdot y-y\cdot x$ defines a Lie $\K$-algebra structure
on $\g$, which is called the {\bf sub-adjacent Lie algebra} of
$(\frkg,\cdot_\frkg)$ and denoted by $\g^c$. Furthermore,
$L:\g\longrightarrow \gl(\g)$ with $x\rightarrow L_x$, where
$L_xy=x\cdot_\g y$, for all $x,y\in \g$, gives a representation of
the Lie algebra $\g^c$ on $\g$. See \cite{Pre-lie algebra in geometry} for more details.

\begin{defi}{\rm(\cite{FMM})}
A {\bf pre-Lie-Rinehart algebra} consists of the following data:
\begin{itemize}
\item[$\bullet$]a unitary commutative $\K$-algebra $A$;
\item[$\bullet$]an $A$-module $E$;
\item[$\bullet$]an $A$-module map $\theta_E:E\longrightarrow\Der(A)$, called the anchor;
\item[$\bullet$]a $\K$-linear pre-Lie operation $\cdot_E:E\otimes E\to E$.
\end{itemize}
These data satisfy the following additional conditions:
\begin{eqnarray}
 X\cdot_E(aY)&=&a(X\cdot_E Y)+\theta_E(X)(a)Y,\\
 (aX)\cdot_E Y&=&a(X\cdot_E Y),\\
 \theta_E(X\cdot_E Y-Y\cdot_E X)&=&[\theta_E(X),\theta_E(Y)]^c,\quad \forall~X,Y\in E,a\in
A.
\end{eqnarray}
We denote a pre-Lie-Rinehart algebra by $(E,\cdot_E,\theta_E)$ if $A$ is fixed.
\end{defi}

For any $X\in E$, we define $L_X:E\longrightarrow E$  and $R_X:E\longrightarrow E$ by
\begin{equation}\label{eq:leftmul}
L_XY=X\cdot_E Y,\quad R_XY=Y\cdot_E X.
\end{equation}

Similar to the proof Theorem 3.2 in \cite{LiuShengBaiChen}, we have
\begin{pro}
Let $(E,\cdot_E,\theta_E)$ be a pre-Lie-Rinehart algebra. Define a skew-symmetric $\K$-bilinear bracket operation $[-,-]_E$ on $E$ by
 $$
 [X,Y]_E=X\cdot_E Y -Y\cdot_E X,\quad \forall ~X,Y\in E.
  $$
Then $(E,[-,-]_E,\theta_E)$ is a Lie-Rinehart algebra, and denoted by
$E^c$, called the {\bf sub-adjacent Lie-Rinehart algebra} of
 $(E,\cdot_E,\theta_E)$. Furthermore, $L$  gives a
  representation of the Lie-Rinehart algebra $E^c$ on $E$.
\end{pro}
\emptycomment{
\begin{proof}
For $\forall~a\in A$, by a direct calculation, we have
\begin{eqnarray*}
[X,aY]_E&=&X\cdot_E(aY) -aY\cdot_E X=a(X\cdot_E Y)+\theta_E(X)(a)Y-aY\cdot_E X\\
&=&a[X,Y]_E+\theta_E(X)(a)Y
\end{eqnarray*}
Then $(E,[-,-]_E,\theta_E)$ is a Lie-Rinehart algebra.

For $\forall~X,Y\in E^c, Z\in E$, we have
\begin{eqnarray*}
L_{[X,Y]_E}Z&=&[X,Y]_E\cdot_E Z=(X\cdot_E Y -Y\cdot_E X)\cdot_E Z\\
&=&X\cdot_E Y\cdot_E Z -Y\cdot_E X\cdot_E Z=L_XL_YZ-L_YL_XZ\\
&=&[L_X,L_Y]^cZ.
\end{eqnarray*}
Then $L$ is a representation of the Lie-Rinehart algebra $E^c$ on $E$.
\end{proof}}

\begin{defi}
Let $(E,\cdot_E,\theta_E)$ and $(F,\cdot_F,\theta_F)$ be pre-Lie-Rinehart
algebras. An $A$-linear map $\varphi:E\longrightarrow F$ is
called a {\bf homomorphism}  of pre-Lie-Rinehart algebras, if the following
conditions are satisfied:
$$\varphi(X \cdot_E Y)=\varphi(X)\cdot_F\varphi(Y),\quad
\theta_F\circ\varphi=\theta_E, \quad\forall X,Y\in E. $$
\end{defi}

It is straightforward to obtain following proposition.
\begin{pro}
Let $(E,\cdot_E,\theta_E)$ and $(F,\cdot_F,\theta_F)$ be pre-Lie-Rinehart algebras, and  $\varphi:E\longrightarrow F$ a homomorphism of pre-Lie-Rinehart algebras. Then $\varphi$ is  a Lie-Rinehart algebra homomorphism from the corresponding sub-adjacent Lie-Rinehart algebras $E^c$ to $F^c$.
\end{pro}
\emptycomment{
\begin{proof}
By a direct calculation, we have
\begin{eqnarray*}
  \varphi([X,Y]_E)&=&\varphi(X \cdot_E Y- Y \cdot_E X)=\varphi(X \cdot_E Y)-\varphi(Y \cdot_E X)\\&=&\varphi(X)\cdot_F\varphi(Y)-\varphi(Y)\cdot_F\varphi(X)\\
  &=&[\varphi(X),\varphi(Y)]_F.
\end{eqnarray*}
 Then $\varphi$ is  a Lie-Rinehart algebra homomorphism from the corresponding sub-adjacent Lie-Rinehart algebras $E^c$ to $F^c$.
\end{proof}}

\emptycomment{\begin{defi}
A pre-Lie-Rinehart algebra $(E,\cdot_E,\theta_E)$ is called a {\bf Novikov-Rinehart algebra} if it also satisfies
\begin{equation}
 ( X\cdot_E Y)\cdot_E Z=(X\cdot_E Z)\cdot_E Y,\quad\forall~X,Y,Z\in E.
\end{equation}
\end{defi}}

\begin{ex}
Any pre-Lie $A$-algebra is a pre-Lie-Rinehart algebra with the anchor $\theta=0$.
\end{ex}

\begin{ex}\label{ex:n-derivation}
  Let $\huaD_n=\{\partial_1,\partial_2,\cdots,\partial_n\}$ be a system of commuting derivations of $A$. We regard the derivations in the endomorphism algebra to be linearly independent. For any $a\in A$, the endomorphisms
  $$a\partial_i:A\longrightarrow A,\quad (a\partial_i)(b)=a\partial_i(b)$$
  are derivations of $A$. Denote by $A\huaD_n = \{\sum_{i=1}^na_i\partial_i\mid a_i\in A,\partial_i\in\huaD\}.$
  Define $\cdot:A\huaD_n\times A\huaD_n\longrightarrow A\huaD_n$ by
  $$(a\partial_i)\cdot (b\partial_j)=a\partial_i(b)\partial_j.$$
  Then $(A\huaD_n,\cdot,\Id)$ is a pre-Lie-Rinehart algebra.
\end{ex}

\begin{ex}
 For the algebra of polynomials $\K[t]$ with $\K=\Real$ or $\Comp$, for all $f\frac{d}{d t},g\frac{d}{d t}\in\Der(\K[t])$, define
  $$f\frac{d}{d t}\cdot g\frac{d}{d t}=fg'\frac{d}{d t},$$
  where $g'=\frac{d g}{d t}$.
 Then $(\Der(\K[t]),\cdot,\Id)$ is a pre-Lie-Rinehart algebra.
\end{ex}

\begin{ex}
   For the algebra of Laurent polynomials $\Comp[z,z^{-1}]$, for all $f\frac{d}{d z},g\frac{d}{d z}\in\Der(\Comp[z,z^{-1}])$ with $f,g\in\Comp[z,z^{-1}]$, define
  $$f\frac{d}{d z}\cdot g\frac{d}{d z}=fg'\frac{d}{d z},$$
  where $g'=\frac{d g}{d z}$.
 Then $(\Der(\Comp[z,z^{-1}]),\cdot,\Id)$ is a pre-Lie-Rinehart algebra.
\end{ex}

\begin{ex}
  Let $(A,\cdot)$ be a communicative associative algebra with a derivation $\partial$. Define
  \begin{eqnarray*}
    x\ast y&=&x\cdot (\partial y),\\
    \theta(x)(y)&=&x\cdot (\partial y),\quad\forall~x,y\in A.
  \end{eqnarray*}
 Then $(A,\ast,\theta)$ is a pre-Lie-Rinehart algebra.
\end{ex}

\begin{ex}
Let $(\huaL,\cdot_\huaL,\theta_\huaL)$ be a left-symmetric algebroid on a manifold $M$. Then $(E=\Gamma(\huaL),\cdot_\huaL,\theta_\huaL)$ is a pre-Lie-Rinehart algebra with $\K=\Real$ and $A=\CWM$ the algebra of smooth functions on a manifold $M$.
\end{ex}

\begin{defi}
Let $(\frkg,\cdot_{\frkg})$ be a pre-Lie $\K$-algebra. An {\bf action} of $\g$ on $A$ is a $\K$-linear map $\lambda:\g\longrightarrow\Der(A)$ satisfying
$$\lambda(x\cdot_{\frkg} y-y\cdot_{\frkg} x)=[\lambda(x),\lambda(y)]^c.$$
\end{defi}

\begin{ex}\label{ex:transformation pre-Lie}
  Let $(\g,\cdot_\g)$ be a pre-Lie $\K$-algebra and $\lambda:\g\longrightarrow\Der(A)$ an action of $\g$ on $A$. Then the {\bf transformation pre-Lie-Rinehart algebra} is $E=A\otimes \g$ with pre-Lie operation
 $$(a\otimes x)\cdot_\lambda (b\otimes y) =ab\otimes x\cdot_\g y+a\lambda(x)(b)\otimes y$$
 and the anchor $\theta_\lambda(a\otimes x)(b)=a\lambda(x)(b)$ for $a,b\in A$ and $x,y\in\g$.
\end{ex}

Obviously, the sub-adjacent Lie-Rinehart algebra of a transformation pre-Lie-Rinehart algebra is a transformation Lie-Rinehart algebra.

\emptycomment{\begin{ex}
  With the notations in Example \ref{ex:pre-Lie action}. Then we obtain a transformation pre-Lie-Rinehart algebra $(\K[t]\otimes\g,\cdot_\lambda,\theta_\lambda)$ corresponding to the case (2) as follows:
   \begin{eqnarray*}
    (f(t)\otimes e_1)\cdot_\lambda (g(t)\otimes e_2)&=&(g(t)\otimes e_2)\cdot_\lambda (f(t)\otimes e_1)=f(t)g(t)\otimes e_4,\\
   (f(t)\otimes e_3)\cdot_\lambda (g(t)\otimes e_1)&=&f(t)g(t)'(t+c)^2\otimes e_1,\\
   (f(t)\otimes e_4)\cdot_\lambda (g(t)\otimes e_1)&=&f(t)g(t)'(t+c)\otimes e_1,\\
   (f(t)\otimes e_2)\cdot_\lambda (g(t)\otimes e_3)&=&2f(t)g(t)\otimes e_1,\\
   (g(t)\otimes e_3)\cdot_\lambda (f(t)\otimes e_2)&=&f(t)g(t)\otimes e_1+g(t)f(t)'(t+c)^2\otimes e_2,\\
    (f(t)\otimes e_4)\cdot_\lambda (g(t)\otimes e_2)&=&-f(t)g(t)\otimes e_2+f(t)g(t)'(t+c)\otimes e_2,\\
     (f(t)\otimes e_3)\cdot_\lambda (g(t)\otimes e_3)&=&f(t)g(t)'(t+c)^2\otimes e_3,\\
     (f(t)\otimes e_3)\cdot_\lambda (g(t)\otimes e_4)&=&f(t)g(t)'(t+c)^2\otimes e_4,\\
    (g(t)\otimes e_4)\cdot_\lambda (f(t)\otimes e_3)&=&f(t)g(t)\otimes e_3+g(t)f(t)'(t+c)\otimes e_3,\\
     (f(t)\otimes e_4)\cdot_\lambda (g(t)\otimes e_4)&=&-f(t)g(t)\otimes e_4+f(t)g(t)'(t+c)\otimes e_4,
  \end{eqnarray*}
  and \begin{eqnarray*}
  \theta_\lambda(1\otimes e_1)(f(t))&=&\theta_\lambda(1\otimes e_2)(f(t))=0,\\
  \theta_\lambda(1\otimes e_3)(f(t))&=&f(t)'(t+c)^2,\theta_\lambda(1\otimes e_4)(f(t))=f(t)'(t+c).
  \end{eqnarray*}
\end{ex}}

Now we give the notion of a derivation of a pre-Lie $A$-algebra, which will be used to construct pre-Lie-Rinehart algebras.
\begin{defi}
  Let $(E,\cdot_E)$ be a pre-Lie $A$-algebra. A {\bf derivation} of $E$ is a pair $(D,\delta)$, where $D:E\longrightarrow E$ is a $\K$-linear operator, $\delta\in\Der(A)$ and they satisfy the conditions:
  \begin{eqnarray*}
    D(X\cdot_E Y)&=&D(X)\cdot_E Y+X\cdot_E D(Y),\\
    D(a X)&=& \delta(a)X+a D(X),\quad\forall~a\in A,X,Y\in E.
  \end{eqnarray*}
\end{defi}

\begin{pro}
  For a derivation $(D,\delta)$ as above, $A\oplus E$ has a natural pre-Lie-Rinehart algebra structure over $A$ given by
  \begin{eqnarray*}
    \theta(a,X)&=&a \delta,\\
    (a,X)\ast (b,Y)&=&(a\delta(b),X\cdot_E Y+a D(Y)),\quad\forall~(a,X),(b,Y)\in A\oplus E.
  \end{eqnarray*}
\end{pro}
\begin{proof}
It follows by a direct calculation.
\end{proof}

\begin{pro}
  Let $(E_i,\cdot_{E_i},\theta_{E_i})$ be two pre-Lie-Rinehart algebras over $A_i$ for $i=1,2$. Consider the $\Comp$-vector space
  $$\huaL=E_1\otimes A_2\oplus E_2\otimes A_1.$$
  Define the operation $\ast:\huaL\times \huaL\longrightarrow \huaL$  by
  \begin{eqnarray*}
    (X_1\otimes a_2)\ast (Y_1\otimes b_2)&=&X_1\cdot_{E_1} Y_1\otimes a_2b_2,\\
     (X_1\otimes a_2)\ast (a_1\otimes X_2)&=&\theta_{E_1}(X_1)(a_1)\otimes a_2X_2,\\
  (a_1\otimes X_2)\ast (X_1\otimes a_2)&=&a_1X_1\otimes\theta_{E_2}(X_2)(a_2),\\
   (a_1\otimes X_2)\ast (b_1\otimes Y_2)&=&a_1b_1\otimes X_2\cdot_{E_2} Y_2
  \end{eqnarray*}
  and the map $\theta_{\huaL}:\huaL\longrightarrow \Der(A_1\otimes A_2)$ by
  \begin{eqnarray*}
    \theta_\huaL(X_1\otimes a_2)(c_1\otimes c_2)&=&\theta_{E_1}(X_1)(c_1)\otimes a_2c_2,\\
    \theta_\huaL(a_1\otimes X_2)(c_1\otimes c_2)&=& a_1c_1\otimes\theta_{E_2}(X_2)(c_2),
  \end{eqnarray*}
  where $X_1,Y_1\in E_1$, $X_2,Y_2\in E_2$, $a_1,b_1,c_1\in A_1$ and $a_2,b_2,c_2\in A_2$. Then $(\huaL,\ast,\theta_\huaL)$ is a pre-Lie-Rinehart algebra over $A_1\otimes A_2$.
\end{pro}
\begin{proof}
 It follows by a direct calculation.
\end{proof}

\emptycomment{We will call the pre-Lie-Rinehart algebra constructed in this way an {\bf extended transformation pre-Lie-Rinehart algebra} of $E$ via a derivation $(D,\delta)$, and it will be denoted by $A\ltimes_{(D,\delta)} E$.}

\section{Form $r$-matrices to pre-Lie-Rinehart algebras }\label{sec:$r$-matrices}
Let $A$ be a commutative $\K$-algebra and  $\huaM$ be an $A$-module. A $\K$-linear map $D:A\longrightarrow \huaM$ is a called a {\bf derivation} if it satisfies the Leibniz rule: $$D(ab)=D(a)b+aD(b),\quad\forall~a,b\in A.$$
$\huaM$-valued derivations form an $A$-module denoted by $\Der(A,\huaM)$.

The functor $\huaM\longrightarrow \Der(A,\huaM)$ is (co)representable: there exists an $A$-module $\Omega^1$, unique up to a unique isomorphism, together with a natural isomorphism of $A$-modules
$$\Der(A,\huaM)\cong \Hom_A(\Omega^1,\huaM).$$
In particular, putting $\huaM=\Omega^1$, the identity map on the right hand side corresponds to the universal derivation
$$ d :A\longrightarrow \Omega^1,$$
 in which $\Omega^1$ can be written by a formal finite sums of terms of the form $a d  b$ for $a,b\in A$ obeying the Leibniz relation
$$ d (ab)= d (a)b+a d (b),\quad\forall~a,b\in A.$$
The module $\frkX^1:=\Der(A)$ forms a Lie algebra under the commutator bracket. By the universal property of $\Omega^1$ one has
$$\frkX^1\cong \Hom_A(\Omega^1,A).$$

Recall that a {\bf Poisson algebra} is a triple $(A,\cdot,\{-,-\})$, where $(A,\cdot)$ is a commutative associative algebra and $(A,\{-,-\})$ is a Lie algebra, such that the Leibniz rule $\{a,b\cdot c\}=\{a,b\}\cdot c+b\cdot\{a,c\}$ holds for all $a,b,c\in A$. For $a,b,u,v\in A$, define  $\pi:\Omega^1\otimes \Omega^1\longrightarrow A$ by
\begin{equation}
  \pi(a d  u\otimes b d  v)=ab\{u,v\}.
\end{equation}
It is not hard to see that $\pi$ is $A$-linear and skew-symmetric.

Define $\pi^\sharp:\Omega^1\longrightarrow \Hom_A(\Omega^1,A)\cong \Der(A)$ by
\begin{equation}
  \pi^\sharp(a  d  u )(v)=a\{u,v\}.
\end{equation}
Define $[-,-]_{\Omega^1}:\Omega^1\times \Omega^1\longrightarrow \Omega^1$ by
\begin{equation}\label{eq:bracket 1-forms}
  [a d  u,b d  v]_{\Omega^1}=a\{u,b\} d  v+b\{a,v\} d  u+ab d \{u,v\}.
\end{equation}
Then Huebschmann in \cite{Hueb1} proved that $(\Omega^1,[-,-]_{\Omega^1}, \pi^\sharp)$ is a Lie-Rinehart algebra and $\pi^\sharp$ is a morphism of Lie-Rinehart algebras between $\Omega^1$ and $\Der(A)$.

Let $(\g,[-,-]_\g)$ be a finite dimensional Lie $\K$-algebra and $\lambda:\g\longrightarrow \Der(A)$ a Lie $\K$-algebra action of $\g$ on $A$. Let $r\in\g\wedge \g$ be an $r$-matrix for the Lie $\K$-algebra $\g$, i.e. for $r=\sum_{i}x_i\wedge y_i$,
\begin{equation}
 \Courant{r,r}:=\sum_{i,j}\big([x_i,x_j]_\g\wedge y_i\wedge y_j+2x_i\wedge[y_i,x_j]_\g\wedge y_j+x_i\wedge x_j\wedge[y_i,y_j]_\g\big)=0.
\end{equation}
Define $\{-,-\}:A\times A\longrightarrow A$ by
\begin{equation}
  \{a,b\}=\sum_i\big(\lambda(x_i)(a)\lambda(y_i)(b)-\lambda(x_i)(b)\lambda(y_i)(a)\big),\quad\forall~a,b\in A.
\end{equation}
Then $(A,\{-,-\})$ is a Poisson algebra. The Jacobi identity follows the  relation:
\begin{eqnarray*}
  \{a,\{b,c\}\}+ \{c,\{a,b\}\}+ \{b,\{c,a\}\}=\half\lambda(\Courant{r,r})(a,b,c),\quad\forall~a,b,c\in A.
\end{eqnarray*}
For $a,b,u,v\in A$, $\pi_r:\Omega^1\otimes \Omega^1\longrightarrow A$ is given by
\begin{equation}
  \pi_r(a d  u\otimes b d  v)=ab\{u,v\}=ab\sum_i\big(\lambda(x_i)(u)\lambda(y_i)(v)-\lambda(x_i)(v)\lambda(y_i)(u)\big).
\end{equation}
$\pi_r^\sharp:\Omega^1\longrightarrow \Hom_A(\Omega^1,A)\cong \Der(A)$ is given by
\begin{equation}
  \pi_r^\sharp(a  d  u )(v)=a\{u,v\}=a\sum_i\big(\lambda(x_i)(u)\lambda(y_i)(v)-\lambda(x_i)(v)\lambda(y_i)(u)\big).
\end{equation}
Then $(\Omega^1,[-,-]_{\Omega^1}, \pi_r^\sharp)$ is a Lie-Rinehart algebra, where $[-,-]_{\Omega^1}$ is given by \eqref{eq:bracket 1-forms}.

Define $\cdot_{\Omega^1}:\Omega^1\times \Omega^1\longrightarrow \Omega^1$ by
\begin{equation}
  (a  d  u)\cdot_{\Omega^1}(b  d  v)=ab\sum_i\big(\lambda(x_i)(u) d \lambda(y_i)(v)-\lambda(y_i)(u) d \lambda(x_i)(v)\big)+a\{u,b\} d  v.
\end{equation}
\begin{pro}
  With the above notations, $(\Omega^1,\cdot_{\Omega^1}, \pi_r^\sharp)$ is a pre-Lie-Rinehart algebra and its sub-adjacent Lie-Rinehart algebra is just $(\Omega^1,[-,-]_{\Omega^1}, \pi_r^\sharp)$.
\end{pro}
\begin{proof}
It is easy to see that the operation $\cdot_{\Omega^1}$ satisfies the conditions (i)-(iii) in definition of pre-Lie-Rinehart-algebra. By direct calculation, for $\alpha,\beta,\gamma\in \Omega^1$ and $a\in A$, we have
\begin{eqnarray*}
(\alpha,\beta,a\gamma)-(\beta,\alpha,a\gamma)=(a\alpha,\beta,\gamma)-(\beta,a\alpha,\gamma)=a\big((\alpha,\beta,\gamma)-(\beta,\alpha,\gamma)\big).
\end{eqnarray*}
Thus we only need to show that $$( d  u, d  v, d  w)=( d  v, d  u, d  w)$$ holds for $u,v,w\in A$. By a similar proof of Theorem 1.2 in \cite{Boucetta}, it follows immediately.
\end{proof}
\begin{rmk}
  In \cite{Boucetta}, Boucetta found that by the action of a Lie algebra on a manifold, an $r$-matrix can induce a flat connection on the manifold. We generalize this result to the pre-Lie-Rinehart algebra structure.
\end{rmk}
\begin{ex}\label{ex:$r$-matrices}
  Consider the standard base of $\sln(2,\K)$ given by
  $$
h =\begin{bmatrix}1& 0\\ 0&-1\end{bmatrix},\quad e =\begin{bmatrix}0& 1\\ 0&0\end{bmatrix},\quad f =\begin{bmatrix}0& 0\\ 1&0\end{bmatrix},
$$
which are related by
$$[h,e]=2e,\quad [h,f]=-2f,\quad[e,f]=h.$$
It is straightforward to check that
$$r= r_1h\wedge e+r_2h\wedge f+r_3e\wedge f$$
 is an $r$-matrix for the Lie algebra $\sln(2,\K)$ if and only if $r_3^2-4r_1r_2=0$.

Define $\lambda:\sln(2,\K)\longrightarrow \Der(\K[x_1,x_2])$ by
$$\lambda(h)=x_1\partial_{x_1}-x_2\partial_{x_2},\quad \lambda(e)=x_1\partial_{x_2},\quad \lambda(f)=x_2\partial_{x_1}.$$
Then $\lambda$ is a Lie algebra action of $\sln(2,\K)$ on $\K[x_1,x_2]$.

The Poisson algebra on $\K[x_1,x_2]$ induced by the Lie algebra action $\lambda$ of $\sln(2,\K)$ on $\K[x_1,x_2]$ is given by
$$\{x_1,x_2\}=r_1x_1^2+r_2x_2^2-r_3x_1x_2,$$
and the corresponding Poisson structure is given by
$$\pi_r=(r_1x_1^2+r_2x_2^2-r_3x_1x_2)\partial_{x_1}\wedge \partial_{x_2}$$
with the condition $r_3^2-4r_1r_2=0$.

Note that  $\Omega^1$ is generated by the basis $\{\dM x_1,\dM x_2\}$ with coefficients in $\K[x_1,x_2]$. The pre-Lie-Rinehart algebra on $\Omega^1$ is determined by
\begin{eqnarray*}
  \dM x_1\cdot_{\Omega^1}  \dM x_2&=&(r_1x_1-r_3x_2)\dM x_1+r_2x_2\dM x_2,\quad \dM x_2\cdot_{\Omega^1}  \dM x_1=-r_1x_1\dM x_1+(-r_2x_2+r_3x_1)\dM x_2,\\
  \pi_r^\sharp(\dM x_1 )&=&(r_1x_1^2+r_2x_2^2-r_3x_1x_2)\partial_{x_2},\quad \pi_r^\sharp(\dM x_2 )=-(r_1x_1^2+r_2x_2^2-r_3x_1x_2)\partial_{x_1},
\end{eqnarray*}
where $r_3^2-4r_1r_2=0$.
\end{ex}

\section{Cohomologies and abelian extensions of pre-Lie-Rinehart algebras}\label{sec:Cohomology and extensions}
In this section, we study cohomologies and abelian extensions of pre-Lie-Rinehart algebras.
\subsection{Cohomologies  of pre-Lie-Rinehart algebras}
\begin{defi}
A {\bf representation} of a pre-Lie-Rinehart algebra $(E,\cdot_E,\theta_E)$ on an $A$-module $\huaE$  consists of a pair
$(\rho,\mu)$, where $(\rho;\huaE)$ is a representation
of the sub-adjacent Lie-Rinehart algebra $E^c$ on $\huaE $ and $\mu:E\longrightarrow \Hom_A(\huaE,\huaE)$ is an $A$-linear
map, such that for all $X,Y\in E$, we have
\begin{eqnarray}\label{representation condition 2}
 \rho(X)\mu(Y)-\mu(Y)\rho(X)=\mu(X\cdot_E Y)-\mu(Y)\mu(X).
\end{eqnarray}
Denote a representation by $(\huaE;\rho,\mu)$.
\end{defi}
Let $(E,\cdot_E,\theta_E)$ be a pre-Lie-Rinehart algebra. It is obvious that if $(\huaE;\rho)$ is a representation of the sub-adjacent Lie-Rinehart algebra $E^c$, then $(\huaE;\rho,0)$ is a representation of pre-Lie-Rinehart algebra $(E,\cdot_E,\theta_E)$.

\begin{pro}
Let $(E,\cdot_E,\theta_E)$ be a pre-Lie-Rinehart algebra and $(\huaE;\rho,\mu)$ its representation. Then $(F=E\oplus \huaE,\ast_F,\theta_F)$ is a pre-Lie-Rinehart algebra, where
$\ast_F$ and $\theta_F$ are given by
\begin{eqnarray*}
(X+u)\ast_F(Y+v)&=&X\cdot_E Y+\rho(X)v+\mu(Y)u,\\
\theta_F(X+u)&=&\theta_E(X),
\end{eqnarray*}
for all $X,Y\in E,u,v\in \huaE$. We denote this pre-Lie-Rinehart algebra by $E\ltimes_{\rho,\mu} \huaE$ and call it the {\bf semidirect product} of $E$ and $\huaE$.
\end{pro}
\emptycomment{We denote this by $E\ltimes_{\rho,\mu} \huaE$.
\begin{proof}
  For all $X,Y\in E,u,v\in \huaE$ and $a\in A$, we have
\begin{eqnarray*}
(X+u)\ast_F(a(Y+v))&=&X\cdot_E(aY)+\rho(X)(av)+\mu(aY)u\\
&=&a(X\cdot_EY)+\theta_E(X)(aY)+a\rho(X)(v)+\theta_E(X)(av)+a\mu(Y)u\\
&=&a(X+u)\ast_F(Y+v)+\theta_F(X)(a)(Y+v).\\
(a(X+u))\ast_F(Y+v)&=&aX\cdot_E Y+\rho(aX)v+\mu(Y)(au)\\
&=&aX\cdot_E Y+a\rho(X)v+a\mu(Y)(u)\\
&=&a(X+u)\ast_F(Y+v).
\end{eqnarray*}
Then $(F=E\oplus \huaE,\ast_F,\theta_F)$ is a pre-Lie-Rinehart algebra.
\end{proof}}

\emptycomment{Let $(E,\cdot_E,\theta_E)$ be a pre-Lie-Rinehart algebra, and
$(\huaE;\rho,\mu)$ be its representation.
Let $(\rho^*;\huaE^*)$ be the dual representation of
$(\rho;\huaE)$, and $\mu^*:E\longrightarrow \Hom(E^*,E^*)$ is defined by
$\langle \mu^*(X) \xi,u\rangle=-\langle \mu(X)u,\xi \rangle,\forall~\xi\in E^*.$

\begin{pro}\label{prop:representation}
With the above notations, we have
\begin{itemize}
\item[$\rm(i)$] $(\huaE;\rho-\mu)$ is a representation of the sub-adjacent Lie-Rinehart algebra $E^c$.
\item[$\rm(ii)$] $(\huaE^*;\rho^*-\mu^*,-\mu^*)$ is a representation of the pre-Lie-Rinehart algebra $(E,\cdot_E,\theta_E)$.
\end{itemize}
\end{pro}}

Now we define the cohomology complex for a pre-Lie-Rinehart algebra $(E,\cdot_E,\theta_E)$ with a representation $(\huaE;\rho,\mu)$. Denote the set of $(n+1)$-cochains by
$$C^{n+1}(E,\huaE)=\Hom_A(\wedge^{n}E\otimes E,\huaE),\
n\geq 0.$$  For all $\varphi\in C^{n}(E,\huaE)$, define $\delta\varphi\in\Hom_A(\otimes^{n+1}E,\huaE)$ is given by
 \begin{eqnarray*}
\delta\varphi(X_1,\cdots,X_{n+1})&=&\sum_{i=1}^{n}(-1)^{i+1}\rho(X_i)\omega(X_1,\cdots,\hat{X_i},\cdots,X_{n+1})\\
 &&+\sum_{i=1}^{n}(-1)^{i+1}\mu(X_{n+1})\omega(X_1,\cdots,\hat{X_i},\cdots,X_n,X_i)\\
 &&-\sum_{i=1}^{n}(-1)^{i+1}\varphi(X_1,\cdots,\hat{X_i},\cdots,X_n,X_i\cdot_E X_{n+1})\\
 &&+\sum_{1\leq i<j\leq n}(-1)^{i+j}\varphi([X_i,X_j]_E,X_1,\cdots,\hat{X_i},\cdots,\hat{X_j},\cdots,X_{n+1})
\end{eqnarray*}
for all $X_i\in E,i=1,\cdots,n+1$.
\begin{pro}
 With the above notations, $\delta$ is a $\K$-linear map and $\delta\varphi\in C^{n+1}(E,\huaE)$ for all $\varphi\in C^{n}(E,\huaE)$. Furthermore, $\delta^2=0$. Thus, we have a well-defined cochain complex $(\bigoplus_{n\geq0}C^{n+1}(E,\huaE),\delta)$ and denote the $n$-th cohomology group of $(\bigoplus_{n\geq0}C^{n+1}(E,\huaE),\delta)$ by $H_{\PRin}^n(E,\huaE)$.
\end{pro}
\begin{proof}
  The proof is similar to the proof of Proposition 5.7 in \cite{LiuShengBaiChen}.
\end{proof}

In the following, we study the relations between the cohomology groups of pre-Lie-Rinehart algebras and Lie-Rinehart algebras.

Given a representation $(\huaE;\rho,\mu)$ on the pre-Lie-Rinehart algebra $(E,\cdot_E,\theta_E)$, for $X\in E$, define $\varrho(X):C^{1}(E,\huaE)\longrightarrow C^{1}(E,\huaE)$ by
\begin{eqnarray}
\varrho(X)(\psi)(Y)&=&\rho(X)\psi(Y)+\mu(Y)\psi(X)-\psi(X\cdot_E Y)
\end{eqnarray}
for $\psi\in C^{1}(E,\huaE)$ and $Y\in E$.
\begin{lem}
  With the above notations, $\varrho$ gives a representation of the sub-adjacent Lie-Rinehart algebra $E^c$ on $C^{1}(E,\huaE)$.
\end{lem}
\begin{proof}
 It is easy to see that
\begin{eqnarray*}
  \varrho(aX)=a\varrho(X),\quad  \varrho(X)(a\psi)=a\varrho(X)(\psi)+\theta_E(X)(a\psi).
\end{eqnarray*}
Next we show that $\varrho([X,Y]_E)=[\varrho(X),\sigma(Y)]^c$. In fact, since $(\huaE;\rho,\mu)$ a representation on the pre-Lie-Rinehart algebra $(E,\cdot_E,\theta_E)$, we have
\begin{eqnarray*}
&&\varrho([X,Y]_E)(\psi)(Z)-[\varrho(X),\sigma(Y)]^c(\psi)(Z)\\
&=&\big(\rho([X,Y]_E-[\rho(X),\rho(Y)]^c\big)(\psi(Z))-\psi\big([X,Y]_E\cdot_E Z+Y\cdot_E  (X\cdot_E Z)\\
&&-X\cdot_E  (Y\cdot_E Z)\big)-\big(\rho(X)\mu(Z)+\mu(Z)\rho(X)-\mu(X\cdot_E Z)+\mu(Z)\mu(X)\big)(\psi(Y))\\
&&+\big(\rho(Y)\mu(Z)+\mu(Z)\rho(Y)-\mu(Y\cdot_E Z)+\mu(Z)\mu(Y)\big)(\psi(X))\\
&=&0.
\end{eqnarray*}
Thus $\varrho$ is a representation of the sub-adjacent Lie-Rinehart algebra $E^c$ on $C^{1}(E,\huaE)$.
\end{proof}

\begin{thm}\label{thm:cohomology of pre-Lie and Lie}
 Let $(E,\cdot_E,\theta_E)$ be a pre-Lie-Rinehart algebra and $(\huaE;\rho,\mu)$ a representation. Define $H:\Omega^n(E,C^{1}(E,\huaE))\longrightarrow C^{n+1}(E,\huaE)$ by
 $$H(\psi)(X_1,\ldots,X_{n+1})=\psi(X_1,\ldots,X_{n})(X_{n+1}),\quad \psi \in\Omega^n(E,C^{1}(E,\huaE)),X_i\in E,$$
 which induces an isomorphism of cochain complexes $\Omega^\ast(E,C^{1}(E,\huaE))$ associated to the representation $\varrho$ of $E^c$ on $C^{1}(E,\huaE)$ and $ C^{\ast}(E,\huaE)$ associated to the representations $(\rho,\mu)$ of $E$ on $\huaE$. In particular,
 $$H_{\PRin}^{n+1}(E,\huaE)\cong H_{\Rin}^n(E^c,C^{1}(E,\huaE)),\quad n>0.$$
\end{thm}
\begin{proof}
First, we need to prove that for any $n\geq 0$ the following diagram is commutative
$$
\xymatrix{
 \ar[rr]^{\dM}\Omega^n(E,C^{1}(E,\huaE)) \ar[d]^{F}
                &&\Omega^{n+1}(E,C^{1}(E,\huaE)) \ar[d]^{F}  \\
 \ar[rr]^{\delta} C^{n+1}(E,\huaE)
               && C^{n+2}(E,\huaE).   }
$$
For $\psi \in\Omega^n(E,C^{1}(E,\huaE))$, we have
\begin{eqnarray*}
 H(\dM \psi)(X_1,\ldots,X_{n+2})&=&(\dM \psi(X_1,\ldots,X_{n+1}))(X_{n+2})\\
  &=&\sum_{i=1}^{n+1}(-1)^{i+1}\big(\varrho(X_i)\psi(X_1\cdots,\widehat{X_i},\cdots,X_{n+1})\big)(X_{n+2})\\
  &&+\sum_{1\leq i<j\leq n+1}(-1)^{i+j}\big(\psi([X_i,X_j]_E,X_1\cdots,\widehat{X_i},\cdots,\widehat{X_j},\cdots,X_{n+1})\big)(X_{n+2})\\
  &=&\sum_{i=1}^{n+1}(-1)^{i+1}\rho(X_i)H(\psi)(X_1\cdots,\widehat{X_i},\cdots,X_{n+1},X_{n+2})\\
  &&+\sum_{i=1}^{n+1}(-1)^{i+1}\mu(X_{n+2})H(\psi)(X_1\cdots,\widehat{X_i},\cdots,X_{n+1},X_{i})\\
  &&-\sum_{i=1}^{n+1}(-1)^{i+1}H(\psi)(X_1\cdots,\widehat{X_i},\cdots,X_{n+1},X_i\cdot_EX_{n+2})\\
  &&+\sum_{1\leq i<j\leq n+1}(-1)^{i+j}H(\psi)([X_i,X_j]_E,X_1\cdots,\widehat{X_i},\cdots,\widehat{X_j},\cdots,X_{n+1},X_{n+2})\\
  &=&\delta H(\psi)(X_1,\ldots,X_{n+2}),
\end{eqnarray*}
which implies that $H\circ \dM \psi=\delta\circ  H(\psi).$ The second conclusion follows immediately.
\end{proof}

\subsection{Abelian  extensions of pre-Lie-Rinehart algebras}
\begin{defi}\label{defi:isomorphic}
\begin{itemize}
\item[\rm(1)] Let $E'$, $E$, $E''$ be pre-Lie-Rinehart algebras over $A$. Assume that the anchor map of pre-Lie-Rinehart algebra $E'$ is trivial. An {\bf abelian extension} of pre-Lie-Rinehart algebras is a short exact sequence of pre-Lie-Rinehart algebras:
$$ 0\longrightarrow E'\stackrel{\imath}{\longrightarrow} E\stackrel{p}\longrightarrow E''\longrightarrow0.$$
We say that $E$ is an abelian extension of $E''$ by $E'$.
\item[\rm(2)]The abelian extension of pre-Lie-Rinehart algebras is called {\bf split} if there exists an $A$-linear map $\sigma:E''\rightarrow E$ such that $p\circ \sigma=\id$.
\item[\rm(3)] Two abelian extensions of $E''$ by $E'$,  $E_1$ and $E_2$,   are said to be {\bf equivalent} if there exists a pre-Lie-Rinehart algebra morphism $\tau:E_2\longrightarrow E_1$ such that we have the following commutative diagram:
\begin{equation}\label{diagram1}
\begin{array}{ccccccccc}
0&\longrightarrow& E'&\stackrel{\imath_2}\longrightarrow&E_2&\stackrel{p_2}\longrightarrow& E''&\longrightarrow&0\\
 &            &\Big\|&       &\tau\Big\downarrow&          &\Big\|& &\\
 0&\longrightarrow&E'&\stackrel{\imath_1}\longrightarrow&E_1&\stackrel{p_1}\longrightarrow&E''&\longrightarrow&0.
 \end{array}
 \end{equation}
\end{itemize}
\end{defi}

 Let $E$ be an abelian extension of $E''$ by $E'$, and  $\sigma:E''\rightarrow E$ a split. Define $\omega:E''\times E''\rightarrow E'$,
 $\rho:E''\longrightarrow\Hom_\K(E',E')$ and $\mu:E''\longrightarrow\Hom_\K(E',E')$ respectively by
\begin{eqnarray}
  \label{eq:str1}\omega(X,Y)&=&\sigma(X)\cdot_{E}\sigma(Y)-\sigma(X\cdot_{E''}Y),\quad \forall X,Y\in E'',\\
 \label{eq:str2} \rho(X)(u)&=&\sigma(X)\cdot_{E} u,\quad \forall X\in E,~u\in E',\\
 \label{eq:str31}\mu(X)(u)&=&u\cdot_{E}\sigma(X),\quad \forall X\in E,~u\in E'.
\end{eqnarray}
It is easy to check that $\omega\in\Hom(E''\otimes_A E'', E')$, $\rho$ and $\mu$ are $A$-linear maps, and $\mu(X)\in\Hom_A(E',E')$.

Given a split $\sigma$, we have $E\cong E''\oplus E'$ as $A$-module, and
the pre-Lie-Rinehart algebra structure on  $E$ can be transferred to  $E''\oplus E'$:
\begin{eqnarray}
(X+u)\ast_\sigma(Y+v)&=&X\cdot_{E''}Y +\omega(X,Y)+\rho(X)(v)+\mu(Y)(u)+u\cdot_{E'} v,\label{bracket of Ext}\\
\theta_\sigma(X+u)&=&\theta_{E''}(X).
\end{eqnarray}

\begin{pro}\label{nonabelian extension of LB}
With the above notations, $(E''\oplus E',\ast_\sigma,\theta_\sigma)$ is a pre-Lie-Rinehart algebra if and only if $\rho,\mu$ and $\omega$ satisfy the following equalities:
\begin{eqnarray}
  \label{eq:extension1}{[\rho(X),\rho(Y)]^c}-\rho([X,Y]_{E''})&=&L'_{\omega(X,Y)-\omega(Y,X)},\\
   \label{eq:extension2} \rho(X)\mu(Y)-\mu(Y)\rho(X)-\mu(X\cdot_{E''} Y)+\mu(Y)\mu(X)&=&R'_{\omega(X,Y)},\\
    \label{eq:extension3}(\rho-\mu)(X)(u)\cdot_{E'}v+u\cdot_{E'}\rho(X)(v)-\rho(X)(u\cdot_{E'}v)&=&0,\\
 \label{eq:extension4}u\cdot_{E'}\mu(X)(v)-v\cdot_{E'}\mu(X)(u)-  \mu(X)([u,v]_{E'})&=&0,
\end{eqnarray}
\begin{eqnarray}
 \label{eq:extension5}\omega(Y,X\cdot_{E''}Z)
-\omega(X,Y\cdot_{E''}Z)+\omega([X,Y]_{E''},Z)&=&\rho(X)\omega(Y,Z)- \rho(Y)\omega(X,Z)\\
\nonumber&-&\mu(Z)\omega(X,Y)+\mu(Z)\omega(Y,X).
\end{eqnarray}
\end{pro}
\begin{proof}
It follows by a direct calculation.
\end{proof}

Now we assume that $E'$ is abelian. Then \eqref{eq:extension1} means that $(\rho,\mu)$ is a representation of $E''$ on $E'$ and \eqref{eq:extension5} means that $\omega$ is a $2$-cocycle of the pre-Lie-Rinehart algebra $E''$  associated to the representation $(E';\rho,\mu)$. Choosing the other split $\sigma':E''\rightarrow E$ of $A$-module, let $\omega'\in\Hom(E''\otimes_A E'', E')$ be the corresponding $2$-cocyle such that
$$\omega'(X,Y)=\sigma'(X)\cdot_{E}\sigma'(Y)-\sigma'(X\cdot_{E''}Y).$$

Set $\varphi=\sigma-\sigma'$. Since $p\circ \sigma=p\circ \sigma'=\id$, we have
\begin{eqnarray*}
  p\circ\varphi= p\circ\sigma-p\circ\sigma'=0,
\end{eqnarray*}
which implies that $\varphi\in \Hom_A(E'',E')$.

Since $E'$ is abelian, thus
\begin{eqnarray*}
\rho(X)(u)&=&\sigma(X)\cdot_{E} u=\sigma'(X)\cdot_{E} u,\\
\mu(X)(u)&=&u\cdot_{E}\sigma(X)=u\cdot_{E}\sigma'(X).
\end{eqnarray*}
This means that the representation $(E';\rho,\mu)$ of $E''$ does not depend on the choice of split.

Then we obtain that
\begin{eqnarray*}
 (\omega-\omega')(X,Y)=\rho(X)\varphi(Y)+\mu(Y)\varphi(X)-\varphi(X\cdot_{E''}Y)=\delta\varphi(X,Y).
\end{eqnarray*}

The classical argument in Eilenberg-Mac Lane cohomology, may easily be extended to a proof of the following:
\begin{thm}\label{thm:secind equivalent classes}
Let $E''$ and $E'$ be pre-Lie-Rinehart algebras. Assume that $E'$ is abelian and $(\rho,\mu)$ is a representation of $E''$ on $E'$. Then the abelian extension $E''$ by $E'$  determined by a $2$-cocyle $\omega\in\Hom(E''\otimes_A E'', E')$ builds a one-to-one corresponding between the equivalent classes of extensions of $E''$ by $E'$ and the cohomology classes in $H_{\PRin}^2(E'',E')$.
\end{thm}

\section{Crossed modules for pre-Lie-Rinehart algebras}\label{sec:crossed module}
In this section, we firs construct a certain transformation pre-Lie-Rinehart algebra, called a free pre-Lie-Rinehart algebra, and show that for $n>1$, the $n$th-cohomology group for this free pre-Lie-Rinehart algebra is trivial.  Then we use the third cohomology group of a pre-Lie-Rinehart algebra to classify the crossed modules for pre-Lie-Rinehart algebras.

\subsection{Free pre-Lie-Rinehart algebras}
Let $(\g,\cdot_\g)$ be a pre-Lie $\K$-algebra and $\g^c$ its sub-adjacent Lie algebra. Let $U(\g^c)$ be the $\K$-enveloping algebra of $\g^c$. Then there is a left module over $U(\g^c)$ on $\g$ given by
$$(x_1\otimes\cdots\otimes x_n)\star y=(L_{x_1}\circ\cdots\circ L_{x_n}) y,\quad\forall~y\in \g, x_1\otimes\cdots\otimes x_n\in U(\g^c) .$$
The following lemma gives a description of free pre-Lie algebra generated by a vector space $V$ instead of rooted trees.
\begin{lem}{\rm(\cite{chapoton})}
  Let $\g$ be a free pre-Lie $\K$-algebra generated by a vector space $V$ and $\sigma:V\longrightarrow \g$ the canonical morphism. Then there is a pre-Lie $\K$-algebra structure on $U(\g^c)\otimes V$ given by
  $$(u,a)\ast (v,b)=((u\star \sigma(a))\otimes v,b),\quad\forall~u,v\in U(\g^c),a,b\in V.$$
 Furthermore, the map $\tilde{\sigma}:U(\g^c)\otimes V\longrightarrow \g$ given by
  $$\tilde{\sigma}(u,a)=u\star \sigma(a)$$
  is an isomorphism of pre-Lie algebras over left $U(\g^c)$-module between $\g$ and $U(\g^c)\otimes V$.
\end{lem}
Let $\g$ be a pre-Lie $\K$-algebra and $M$ a $\g$-module. Then we have the cochain complex $(\bigoplus_{n\geq0}C_\K^{n+1}(\g,M),\delta)$ with
$$C^{n+1}_\K(\g,M)=\Hom_\K(\wedge^{n}\g\otimes \g,M),\
n\geq 0.$$
We denote the corresponding $n$th-cohomology group by $H_{\PLie}^n(\g,M)$ and the $n$th-cohomology group of the sub-adjacent Lie $\K$-algebra $\g^c$ by $H_{\Lie}^n(\g^c,M)$.
\begin{pro}{\rm(\cite{chapoton})}\label{pro:free pre-Lie 2}
  Let $\g$ be a free pre-Lie $\K$-algebra generated by a vector space $V$ and $M$ a pre-Lie $\K$-algebra $\g$-module. Then
 \begin{equation*}
 H^n_{\PLie}(\g,M)=\left\{\begin{array}{rcll}
&V & n=1\\
&0,& n\neq 1.
   \end{array}\right.
\end{equation*}
\end{pro}

\emptycomment{
\begin{proof}
It follows that
 \begin{eqnarray*}
H^n_{\PLie}(\g,M)&=&Tor^{n-1}_{U(\g^c)}(U(\g^c)\otimes V,M)\\
&=&\left\{\begin{array}{rcll}
&V,& n=1\\
&0,& n\neq 1.
   \end{array}\right.
\end{eqnarray*}
\end{proof}}

Given a vector space $V$ over $\K$ and $\varphi:V \longrightarrow \Der(A)$ a $\K$-linear map. Let $\g$ be the free pre-Lie $\K$-algebra generated by the vector space $V$. Then we have the unique action of $U(\g^c)\otimes V$ on $A$ $\Phi:U(\g^c)\otimes V\longrightarrow \Der(A)$ which extends the map $\varphi$ by
$$\Phi(u,x)=\varphi\circ\tilde{\sigma}(u,x),\quad\forall~(u,x)\in U(\g^c)\otimes V.$$
Applying the construction in Example \ref{ex:transformation pre-Lie}, we get a transformation pre-Lie-Rinehart algebra on $A\otimes (U(\g^c)\otimes V)$. We call this pre-Lie-Rinehart algebra {\bf the free  pre-Lie-Rinehart algebra generated by $\varphi$}. In particular, let $(E,\cdot_E,\theta_E)$ be a pre-Lie-Rinehart algebra, then there is a free pre-Lie-Rinehart algebra generated by the pre-Lie-Rinehart algebra $E$ through the map $\theta_E:E\longrightarrow \Der(A)$ and denote by $\frkE$. Furthermore, $\pi:\frkE\longrightarrow E$ given by
$$\pi(a\otimes (u,X))=a(u\star \sigma(X)),\quad\forall~a\in A,(u,X)\in U(E^c)\otimes E$$
is a surjective morphism of pre-Lie-Rinehart algebras.

By a similar proof of Lemma 2.9 in \cite{CaLaPi1}, we have
\begin{lem}\label{lem:free pre-Lie}
  Let $(\g,\cdot_\g)$ be a pre-Lie $\K$-algebra and $\rho:\g\longrightarrow\Der(A)$ an action of $\g$ on $A$. Let $E=A\otimes \g$ be the transformation pre-Lie-Rinehart algebra given by Example \ref{ex:transformation pre-Lie}. Then for any pre-Lie-Rinehart $E$-module $\huaE$, we have the canonical isomorphism of cochain complexes $C^{n}(E,\huaE)\cong C^{n}_\K(\g,\huaE)$ for $n\geq1$ and thus
  $$H_{\PRin}^n(E,\huaE)\cong H_{\PLie}^n(\g,\huaE),\quad n\geq1.$$
\end{lem}

By Lemma \ref{lem:free pre-Lie} and Proposition \ref{pro:free pre-Lie 2}, we have
\begin{pro}
  Let $\frkV$ be the  free  pre-Lie-Rinehart algebra generated by $\varphi:V \longrightarrow \Der(A)$ and $\huaE$ any pre-Lie-Rinehart module over $\frkV$. Then
  $$H^n_{\PRin}(\frkV,\huaE)=0,\quad n>1.$$
\end{pro}

\begin{cor}
 Let $(E,\cdot_E,\theta_E)$ be a pre-Lie-Rinehart algebra and $\frkE$ be the free pre-Lie-Rinehart algebra generated by $\theta_E$ and let $\huaE$ be any pre-Lie-Rinehart module over $\frkE$. Then
  $$H^n_{\PRin}(\frkE,\huaE)=0,\quad n>1.$$
\end{cor}

\subsection{Crossed modules for pre-Lie-Rinehart algebras}
\begin{defi}\label{defi:cmodPLRA}
  A {\bf crossed module for pre-Lie-Rinehart algebras} over $A$ consists of a pre-Lie-Rinehart algebra $(E,\cdot_{E},\theta_{E})$ and a pre-Lie $\K$-algebra
  $(\huaE,\cdot_{\huaE})$ together with a representation $(\rho,\mu)$ of $E$ on $\huaE$ and $A$-linear pre-Lie algebra homomorphism $\partial:\huaE\longrightarrow E$ such that the following identities hold:
  \begin{itemize}
\item[\rm(1)]$\partial(\rho(X)u)=X\cdot_{E} \partial(u),\quad \partial(\mu(X)u)=\partial(u)\cdot_{E}X;$
\item[\rm(2)]$\rho(\partial(u))v=\mu(\partial(v))u=u\cdot_{\huaE} v$;
\item[\rm(3)]$\theta_E\circ\partial=0$
\end{itemize}
for all $u,v\in\huaE,X\in E,a\in A$. We denote a crossed module for pre-Lie-Rinehart algebras by $(E,\huaE,\partial,(\rho,\mu))$.
\end{defi}

We call an $A$-submodule $F$ of $(E,\cdot_E,\theta_E)$ a {\bf pre-Lie-Rinehart subalgebra} if $F$ is a pre-Lie $\K$-subalgebra of $(E,\cdot_E)$ and $\theta_F(X)=\theta_E(X)$ for $X\in F$. We call a pre-Lie-Rinehart subalgebra $F$ of $E$ an {\bf ideal} if $F$ is an ideal of $E$ as pre-Lie $\K$-algebra and $\theta_F=0$.

\begin{ex}
  Let $(E,\cdot_E,\theta_E)$ a pre-Lie-Rinehart algebra and $F$ an ideal of $E$. Then $(E,F,\partial=\imath,(\rho,\mu))$ is a crossed module for pre-Lie-Rinehart algebras, where $\imath:F\longrightarrow E$ is the inclusion, and $(\rho,\mu)$ is given by $\rho(X)(Y)=X\cdot_E Y,~\mu(X)(Y)=Y\cdot_E X$ for $X\in E$ and $Y\in F$.
\end{ex}

\begin{ex}
  For any pre-Lie-Rinehart algebra homomorphism $f:E\longrightarrow E'$, $\Ker~ f$ is an ideal of $E$. Then $(E,\Ker ~f,\partial=\imath,(\rho,\mu))$ is a crossed module for pre-Lie-Rinehart algebras, where $\imath:\Ker ~f\longrightarrow E$ is the inclusion, and $(\rho,\mu)$ is given by $\rho(X)(Y)=X\cdot_E Y,~\mu(X)(Y)=Y\cdot_E X$ for $X\in E$ and $Y\in \Ker f$.
\end{ex}

\begin{ex}
  Let $(\rho,\mu)$ be a representation of pre-Lie-Rinehart algebra $E$ on $\huaE$, then $(E,\huaE,\partial=0,(\rho,\mu))$ is a crossed module for pre-Lie-Rinehart algebras.
\end{ex}

\begin{pro}
Let $(E,\huaE,\partial,(\rho,\mu))$ be a crossed module for pre-Lie-Rinehart algebras over $A$. Then we have
\begin{eqnarray*}
  \rho(X)(u\cdot_\huaE v)&=& \rho(X)(u)\cdot_\huaE v+u\cdot_\huaE \rho(X)(v)-(\mu(X)u)\cdot_\huaE v,\\
  \mu(X)(u\cdot_\huaE v)&=& \mu(X)(v\cdot_\huaE u)+u\cdot_\huaE \mu(X)(v)-v\cdot_\huaE(\mu(X)u),\quad\forall~X\in E,u,v\in\huaE.
\end{eqnarray*}
Consequently, there is a pre-Lie-Rinehart algebra structure on $E\oplus\huaE$ given by
\begin{eqnarray*}
  (X+u)\ast(Y+v)&=&X\cdot_E Y+\rho(X)v+\mu(Y)u+u\cdot_\huaE v,\\
  \theta(X+u)&=&\theta_E(X).
\end{eqnarray*}
\end{pro}
\begin{proof}
By the definition of crossed module for pre-Lie-Rinehart algebras, for $X\in E,u,v\in\huaE$, we have
\begin{eqnarray*}
  \rho(X)(u\cdot_\huaE v)&=& \rho(X)(\rho(\partial(u))v)=[\rho(X),\rho(\partial(u))]^cv+\rho(\partial(u))\rho(X)v\\
  &=&\rho([X,\partial(u)]_E)v+\rho(\partial(u))\rho(X)v\\
  &=&\rho(X\cdot_E\partial(u))v-\rho(\partial(u)\cdot_EX)v+\rho(\partial(u))\rho(X)v\\
  &=&\rho(\partial(\rho(X)u))v-\rho(\partial(u)\cdot_EX)v+\rho(\partial(u))\rho(X)v\\
  &=&\rho(X)(u)\cdot_\huaE v+u\cdot_\huaE \rho(X)(v)-(\mu(X)u)\cdot_\huaE v.
\end{eqnarray*}
By a similar calculation, we have
$$\mu(X)(u\cdot_\huaE v)= \mu(X)(v\cdot_\huaE u)+u\cdot_\huaE \mu(X)(v)-v\cdot_\huaE(\mu(X)u).$$
It is straightforward to check that $(E\oplus\huaE,\ast)$ is a pre-Lie algebra. Furthermore, we have
\begin{eqnarray*}
  (X+u)\ast(a(Y+v))&=&X\cdot_E (aY)+\rho(X)(av)+\mu(aY)u+u\cdot_\huaE (av)\\
  &=&a(X\cdot_E Y)+\theta_E(X)(a)Y+a\rho(X)(v)+\theta_E(X)(a)v+a\mu(Y)u+a(u\cdot_\huaE v)\\
  &=&a(X\cdot_E Y+\rho(X)(v)+\mu(Y)u+u\cdot_\huaE v)+\theta_E(X)(a)Y+\theta_E(X)(a)v\\
  &=&a\big( (X+u)\ast(Y+v))+\theta(X+u)(a)(Y+v),\\
  (a(X+u))\ast(Y+v)&=&a(X\cdot_E Y)+\rho(aX)v+\mu(Y)(au)+a(u\cdot_\huaE v)\\
  &=&a(X\cdot_E Y)+a(\rho(X)v)+a(\mu(Y)(u))+a(u\cdot_\huaE v)\\
  &=&a\big((X+u)\ast(Y+v)\big).
\end{eqnarray*}
Thus, $(E\oplus\huaE,\ast,\theta)$ is a pre-Lie-Rinehart algebra over $A$.
\end{proof}

In the following, we recall the definition of crossed module for Lie-Rinehart algebras.
\begin{defi}{\rm(\cite{CaLaPi1})}
  A {\bf crossed module for Lie-Rinehart algebras} over $A$ consists of a Lie-Rinehart algebra $(E,\cdot_{E},[-,-]_{E})$ and a Lie $\K$-algebra
  $(\huaE,[-,-]_{\huaE})$ together with a representation $\rho$ of $E$ on $\huaE$ and $A$-linear Lie algebra homomorphism $\partial:\huaE\longrightarrow E$ such that the following identities hold:
  \begin{itemize}
\item[\rm(1)]$\partial(\rho(X)u)=[X,\partial(u)]_{E};$
\item[\rm(2)]$\rho(\partial(u))v=[u, v]_{\huaE}$;
\item[\rm(3)]$\theta_E\circ\partial=0$;
\item[\rm(4)]$\rho(X)[u,v]_\huaE=[\rho(X)(u),v]_\huaE+[u,\rho(X)(v)]_\huaE$
\end{itemize}
for all $u,v\in\huaE,X\in E,a\in A$. We denote a crossed module for Lie-Rinehart algebras by $(E,\huaE,\partial,\rho)$.
\end{defi}

\begin{pro}
 Let $(E,\huaE,\partial,(\rho,\mu))$ be a crossed module for pre-Lie-Rinehart algebras over $A$. Then $(E^c,\huaE^c,\partial,\rho-\mu)$ is a crossed module for Lie-Rinehart algebras, where $E^c$ and $\huaE^c$  are the corresponding sub-adjacent Lie-Rinehart algebras of $E$ and $\huaE$, respectively.
\end{pro}
\begin{proof}
It follows by a direct calculation.
\end{proof}

Let $(E,\huaE,\partial,(\rho,\mu))$ be a crossed module for pre-Lie-Rinehart algebras over $A$. Then by (1) in definition of crossed module for pre-Lie-Rinehart algebras, $\img~\partial$ is a pre-Lie $\K$-ideal of $E$ and an $A$-submodule, therefore,  by (3) in definition of crossed module for pre-Lie-Rinehart algebras, $\coker~\partial$ has a natural structure of pre-Lie-Rinehart algebra. Furthermore, by (2) and $\partial$ is $A$-linear in definition of crossed module for pre-Lie-Rinehart algebras, $\Ker~\partial$ is an abelian $A$-ideal of $\huaE$ and the representation of $E$ on $\huaE$ gives rise to a pre-Lie-Rinehart module of $\coker\partial$ on $\Ker\partial$.

Let $F$ be a pre-Lie-Rinehart algebra and let $\huaF$ be a pre-Lie-Rinehart module over $F$. A {\bf crossed extension of $F$ by $\huaF$} is an exact sequences of pre-Lie-Rinehart algebras
$$ 0\longrightarrow \huaF\stackrel{\imath}{\longrightarrow}\huaE\stackrel{\partial}{\longrightarrow}E\stackrel{p}\longrightarrow F\longrightarrow0,$$
where $\partial:\huaE\longrightarrow E$ is a crossed module of pre-Lie-Rinehart algebras over $A$ and the canonical maps $\coker~\partial\longrightarrow F$ and $\huaF\longrightarrow \Ker~\partial$ are isomorphism of pre-Lie-Rinehart algebras and modules respectively.

Two crossed extensions of $F$ by $\huaF$ with respect to the crossed modules $\partial:\huaE\longrightarrow E$ and $\partial':\huaE'\longrightarrow E'$, respectively, are said to be {\bf equivalent} if there exist pre-Lie-Rinehart algebras homomorphisms $\phi:\huaE\longrightarrow \huaE'$ and  $\psi:E\longrightarrow E'$ such that
\begin{equation}\label{eq:isomorphism conds}
  \phi(\rho_E(X)u)=\rho_{E'}(\psi(X))\phi(u),\quad  \phi(\mu_E(X)u)=\mu_{E'}(\psi(X))\phi(u),\quad\forall~X\in E,u\in \huaE
\end{equation}
and the following diagram is communicative
\begin{equation}\label{diagram2}
\begin{array}{ccccccccccc}
0&\longrightarrow& \huaF&\stackrel{\imath}\longrightarrow&\huaE&\stackrel{\partial}\longrightarrow&E&\stackrel{p}\longrightarrow& F&\longrightarrow&0\\
 &            &\Big\|&       &\phi\Big\downarrow&      &\psi\Big\downarrow&      &\Big\|& &\\
 0&\longrightarrow&\huaF&\stackrel{\imath'}\longrightarrow&\huaE'&\stackrel{\partial'}\longrightarrow&E'&\stackrel{p'}\longrightarrow&F&\longrightarrow&0.
 \end{array}\end{equation}
Let $\crmod(F,\huaF)$ denote the set of equivalence classes of crossed extensions for the pre-Lie-Rinehart algebra $F$ by $\huaF$.

\begin{thm}\label{crossed module of pLA}
 Let $F$ be a pre-Lie-Rinehart algebra and $\huaF$  a pre-Lie-Rinehart module over $F$. Then there exists a natural bijection
  $$\mu:\crmod(F,\huaF)\longrightarrow H_{\PRin}^3(F,\huaF).$$
\end{thm}
\begin{proof}
We define $\mu:\crmod(F,\huaF)\longrightarrow H_{\PRin}^3(F,\huaF)$ as follows. Consider the following crossed extension of $F$ by $\huaF$
$$ 0\longrightarrow \huaF\stackrel{\imath}{\longrightarrow}\huaE\stackrel{\partial}{\longrightarrow}E\stackrel{p}\longrightarrow F\longrightarrow0.$$
Set $N=\img\partial=\Ker~p$. By our assumption, we can choose $A$-linear sections $s:F\longrightarrow E$ and $\sigma:N\longrightarrow \huaE$. Thus $p\circ s=\id_F$ and $\partial\circ\sigma=\id_N$. One defines $g:F\otimes F\longrightarrow \huaE$ by
$$g(X,Y)=\sigma\big(s(X)\cdot_Es(Y)-s(X\cdot_F Y)\big),\quad\forall~X,Y\in F.$$
It is easy to see that $g(X,aY)=ag(X,Y)$. By a direct calculation, we have
$$g(aX,Y)-ag(X,Y)=\big(\theta_E\circ s-\theta_F\big)(X)(a)s(Y).$$
Sine $p$ is a pre-Lie-Rinehart algebra morphism, one has $\theta_E(Z)=\theta_F\circ p(Z)$ for $Z\in E$. Thus $\theta_E\circ s(X)=\theta_F(X)$ and then $g(aX,Y)=ag(X,Y)$. Therefor, $g$ is an $A$-linear map.

Next one defines $f:F\otimes F\otimes F\longrightarrow \huaE$ as following:
\begin{eqnarray*}
f(X,Y,Z)&=&\rho_E(s(X))g(Y,Z)- \rho_E(s(Y))g(X,Z)-\mu_E(s(Z))g(X,Y)+\mu_E(s(Z))g(Y,X)\\
&&-g([X,Y]_F,Z)+g(X,Y\cdot_{F}Z)-g(Y,X\cdot_{F}Z),\quad\forall~X,Y,Z\in F.
\end{eqnarray*}
Using properties of crossed module, it is straightforward to check that
$$f(X,Y,Z) =-f(Y,X,Z),\quad \partial(f(X,Y,Z))=0,$$
and $f(aX,Y,Z)=f(X,aY,Z)=f(X,Y,aZ)=af(X,Y,Z).$ Thus $f\in\Hom_A(\wedge^2 F\otimes F,\huaF)$.

Note that the representation $(\rho_F,\mu_F)$ of $F$ on $\huaF$ satisfies
$$\rho_F(X)(u)=\rho_E(s(X))(u),\quad\mu_F(X)(u)=\mu_E(s(X))(u),\quad\forall~X\in F,u\in\huaF. $$
It is a routine to check that $\delta f(X,Y,Z,W)=0 $ for $X,Y,Z,W\in F$ associated to the representation $(\huaF;\rho_F,\mu_F)$. Thus $f$ is a $3$-cocycle.

Second, we prove that the cohomology class of $f$ does not depend on the $A$-sections $s$.
Let us choose the other $A$-linear section $\tilde{s}:F\longrightarrow E$ and $\tilde{f}$ is the corresponding $3$-cocycle using $\tilde{s}$ instead of $s$, in which $\tilde{g}(X,Y)=\sigma\big(\tilde{s}(X)\cdot_E\tilde{s}(Y)-\tilde{s}(X\cdot_F Y)\big)$. Since $s$ and $\tilde{s}$ are $A$-sections of $p$, there exists an $A$-linear map $\varphi:F\longrightarrow \huaE$ such that
$$\partial\circ \varphi=\tilde{s}-s.$$
By direct calculation, we have
\begin{eqnarray*}
  (\tilde{f}-f)(X,Y,Z)&=&\rho_E(s(X))(\tilde{g}-g)(Y,Z)- \rho_E(s(Y))(\tilde{g}-g)(X,Z)-\mu_E(s(Z))(\tilde{g}-g)(X,Y)\\
  &&+\mu_E(s(Z))(\tilde{g}-g)(Y,X)-(\tilde{g}-g)([X,Y]_F,Z)+(\tilde{g}-g)(X,Y\cdot_{F}Z)\\
&&-(\tilde{g}-g)(Y,X\cdot_{F}Z)+\varphi(X)\cdot_\huaE \tilde{g}(Y,Z)-\varphi(Y)\cdot_\huaE \tilde{g}(X,Z)\\
&&-\tilde{g}(X,Y)\cdot_\huaE\varphi(Z)+\tilde{g}(Y,X)\cdot_\huaE\varphi(Z).
\end{eqnarray*}
Define $\omega:F\otimes F\longrightarrow \huaE$ by
$$\omega(X,Y)=\rho_E(\tilde{s}(X))\varphi(Y)+\mu_E(\tilde{s}(Y))\varphi(X)-\varphi(X\cdot_F Y)+\varphi(X)\cdot_\huaE\varphi(Y). $$
Using properties of crossed module, we have
$$\partial\circ \omega=\partial\circ(\tilde{g}-g).$$
Thus $\tilde{g}-g-\omega\in\Hom_A( F\otimes F,\huaF)$. Using properties of crossed module and the definition of pre-Lie-Rinehart algebra and its representation, we have
$$\delta(\tilde{g}-g-\omega)(X,Y,Z)= (\tilde{f}-f)(X,Y,Z).$$
Thus the cohomology class of $f$ does not depend on the $A$-sections $s$.

Let $\phi:\huaE\longrightarrow \huaE'$ and  $\psi:E\longrightarrow E'$ be the equivalent    homomorphisms  from the  crossed extension $\partial:\huaE\longrightarrow E$ to the crossed extension $\partial':\huaE'\longrightarrow E'$. Assume that  $s:F\longrightarrow E$ and $\sigma:N\longrightarrow \huaE$ are $A$-linear sections of $p$ and $\partial$, respectively and  $s':F'\longrightarrow E'$ and $\sigma':N'\longrightarrow \huaE'$ are $A$-linear sections of $p'$ and $\partial'$, respectively. Let $f$ and $f'$ denote the corresponding $3$-cocyles. Since $p'\circ \psi\circ s=\id_F$, $\psi\circ s$ is the other $A$-linear section of $p'$ and thus we can replace $s'$ by $\psi\circ s$ to define $f'$. Define $\varpi:F\otimes F\longrightarrow \huaF$ by
$$\varpi(X,Y)=(\phi\circ \sigma-\sigma'\circ \psi)(s(X)\cdot_Es(Y)-s(X\cdot_F Y)),\quad\forall~X,Y\in F.$$
By \eqref{eq:isomorphism conds}, we have
\begin{eqnarray*}
  f(X,Y,Z)&=&\phi(f(X,Y,Z))\\
  &=&\phi(\rho_E(s(X))g(Y,Z))- \phi(\rho_E(s(Y))g(X,Z))-\phi(\mu_E(s(Z))g(X,Y))\\
&&+\phi(\mu_E(s(Z))g(Y,X))-\phi(g([X,Y]_F,Z))+\phi(g(X,Y\cdot_{F}Z))-\phi( g(Y,X\cdot_{F}Z))\\
&=&\rho_{E'}(\psi (s(X)))\phi(g(Y,Z))-\rho_{E'}(\psi( s(Y)))\phi( g(X,Z))-\mu_{E'}(\psi(s(Z)))\phi(g(X,Y))\\
&&+\mu_{E'}(\psi(s(Z)))\phi(g(Y,X))-\phi(g([X,Y]_F,Z))+\phi( g(X,Y\cdot_{F}Z))-\phi(g(Y,X\cdot_{F}Z)).
\end{eqnarray*}
By the fact that $\phi$ and $\psi$ are the pre-Lie-Rinehart algebra homomorphisms, it is not hard to see that
$$(f-f')(X,Y,Z)=\delta\varpi(X,Y,Z).$$
This proves that the cohomology class of $f$ does not depend on the equivalence classes of crossed modules for pre-Lie-Rinehart algebras $\partial:\huaE\longrightarrow E$  and also the class of $f$ does not depend on the $A$-section $\sigma$.

Conversely, define $\nu:H_{\PRin}^3(F,\huaF)\longrightarrow\crmod(F,\huaF)$ as follows. Let $\frkF$ be the free pre-Lie-Rinehart algebra generated by $F$, then we have the following exact sequence
$$ 0\longrightarrow K\stackrel{\imath}{\longrightarrow} \frkF\stackrel{p}\longrightarrow F\longrightarrow0.$$
Furthermore, given a representation $(\rho_F,\mu_F)$ of pre-Lie-Rinehart algebra $F$ on $\huaF$, then there is a natural representation $(\rho_\frkF,\mu_\frkF)$ of $\frkF$ on $F$ and then a trivial representation $(\rho_K=0,\mu_K=0)$ of $K$ on $\huaF$ as follows:
$$\rho_\frkF(X)u=\rho_F(p(X))u,\quad \mu_\frkF(X)u=\mu_F(p(X))u,\quad X\in \frkF,u\in F.$$
The surjective map $p:\frkF\longrightarrow F$ induces a map of cochain complexes:
\begin{equation}\label{diagram1}
\begin{array}{ccccccccc}
\cdots&\stackrel{\delta}\longrightarrow&C^{2}(F,\huaF)&\stackrel{\delta}\longrightarrow& C^{3}(F,\huaF)&\stackrel{\delta}\longrightarrow&\cdots\\
&       &p^*\Big\downarrow&          &p^*\Big\downarrow& &\\
\cdots&\stackrel{\delta}\longrightarrow&C^{2}(E,\huaF)&\stackrel{\delta}\longrightarrow&C^{3}(E,\huaF)&\stackrel{\delta}\longrightarrow&\cdots.
 \end{array}\end{equation}
 By the fact that $H_{\PRin}^3(\frkF,\huaF)=0$, for any $3$-cocycle $\kappa\in C^{3}(F,\huaF)$, there exists a $2$-cochain $\omega\in C^{2}(\frkF,\huaF)$ such that $\delta\omega=p^*\kappa$. Let $\huaE=K\oplus \huaF$ as $A$-module. Define
\begin{eqnarray}
 ( X_1+u_1)\cdot_\huaE (X_2+u_2)&=&X_1\cdot_E X_2+\omega(X_1,X_2)\\
\tilde{ \rho}_\frkF(Y)(X+u)&=&Y\cdot_E X+\rho_E(Y)u+\omega(Y,X),\\
  \tilde{\mu}_\frkF(Y)(X+u)&=&X\cdot_E Y+\mu_E(Y)u+\omega(X,Y),\\
  \partial(X+u)&=&X,
\end{eqnarray}
for all $X,X_1,X_2\in K,Y\in E,u,u_1,u_2\in\huaF.$ It is not hard to check that $\partial:\huaE\longrightarrow \frkF$ with the representation $( \tilde{\rho}_\frkF, \tilde{\mu}_\frkF)$ of $\frkF$ on $\huaE$ given above is a crossed module.

In the following, we have to check that the $\partial:\huaE\longrightarrow \frkF$ constructed above does not depend on the chosen $2$-cochain $\omega$ and the representative of the cohomology class of $\kappa$.

Let $\omega'\in C^{2}(\frkF,\huaF)$ be the other $2$-cochain such that $\delta\omega'=p^*\kappa$ and denoted by $\partial':\huaE'\longrightarrow \frkF$ with the representation $( \tilde{\rho}_{\frkF}, \tilde{\mu}_{\frkF})$ of $\frkF$ on $\huaE'$ the corresponding crossed module. Then $\omega-\omega'$ is a $2$-cocycle of $C^{2}(\frkF,\huaF)$.  By the fact that $H_{\PRin}^2(\frkF,\huaF)=0$, then there exists a $1$-cochain $\varphi\in C^{1}(\frkF,\huaF) $ such that $\delta \varphi=\omega-\omega'$. It is easy to check that for $X\in K,u\in \huaF,Y\in \frkF$, $\phi(X+u)=X+\varphi(X)+u$ and $\psi(Y)=Y$ give a homomorphism from the crossed module $\partial:\huaE\longrightarrow \frkF$ to $\partial':\huaE'\longrightarrow \frkF$. Thus these two crossed modules are in the same equivalence class.

Let $\kappa'\in C^{3}(F,\huaF)$ be the other $3$-cocycle cohomologous to $\kappa$, then there exists a  $2$-cocycle $\varpi\in C^{2}(F,\huaF)$ such that $\delta \varpi=\kappa'-\kappa$. For any $\omega\in C^{2}(\frkF,\huaF)$ with $\delta\omega=p^*\kappa$, we have $\delta(\omega+p^*\varpi)=p^* \kappa'$. Thus we can choose $\omega'=\omega+p^*\varpi$ to define a crossed module $\partial':\huaE'\longrightarrow \frkF$. It is easy to check that for $X\in K,u\in \huaF,Y\in E$, $\phi(X+u)=X+u$ and $\psi(Y)=Y$ give a homomorphism from crossed module $\partial:\huaE\longrightarrow \frkF$ to $\partial':\huaE'\longrightarrow \frkF$. Thus these two crossed modules are in the same equivalence class.

It is straightforward to check that the maps $\mu$ and $\nu$ are inverse to each other.
\end{proof}

\section{Crossed modules for (pre-)Lie-Rinehart algebras and (pre-)Lie-Rinehart $2$-algebras}\label{sec:Lie-Rinehart $2$-algebras}
In this section, we first give the notions of (pre-)Lie-Rinehart $2$-algebras. Next, we establishe a
one-to-one correspondence between strict (pre-)Lie-Rinehart $2$-algebras and crossed modules for (pre-)Lie-Rinehart algebras.
\begin{defi}\label{defi:2Lie}{\rm(\cite{BC})}
  A   {\bf Lie $\K$-$2$-algebra} $\huaV$ over $\K$ consists of the following data:
\begin{itemize}
\item[$\bullet$] a complex of vector spaces over $\K:V_{1}\stackrel{l_1}{\longrightarrow}V_0,$

\item[$\bullet$] a skew-symmetric bilinear map $l_2:V_{i}\otimes V_{j}\longrightarrow
V_{i+j}$, where  $0\leq i+j\leq1$,

\item[$\bullet$] a  skew-symmetric trilinear map $l_3:\wedge^3V_0\longrightarrow
V_{1}$,
   \end{itemize}
   such that for all $x_i,x,y,z\in V_0$ and $u,v\in V_{1}$, the following equalities are satisfied:
\begin{itemize}
\item[$\rm(a)$] $l_1 l_2(x,u)=l_2(x,l_1(u)),\quad l_2(l_1( u),v)=l_2( u,l_1(v)),$
\item[$\rm(b)$]$l_1 l_3(x,y,z)=l_2(x,l_2(y,z))+l_2(z,l_2(x,y))+l_2(y,l_2(z,x)),$
\item[$\rm(c)$]$ l_3(x,y,l_1(u))=l_2(x,l_2(y,u))+l_2(u,l_2(x,y))+l_2(y,l_2(u,x)),$
\item[$\rm(d)$] the Jacobiator identity:
\begin{eqnarray*}
&&\sum_{i=1}^4(-1)^{i+1}l_2(x_i,l_3(x_1,\cdots,\hat{x_i},\cdots,x_4))\\
&&+\sum_{i<j}l_3(l_2(x_i,l_j),x_1,\cdots,\hat{x_i},\cdots,\hat{x_j},\cdots,x_4)=0.\end{eqnarray*}
   \end{itemize}
\end{defi}
We usually denote a Lie $\K$-$2$-algebra by
$(V_{1},V_0,l_1,l_2,l_3)$, or simply by
$\huaV$.

\begin{defi}\label{defi:LR 2-A}
A {\bf Lie-Rinehart $2$-algebra} over $A$ is a pair that consists of a Lie $2$-$\K$-algebra $(L_0,L_1,l_1,l_2,l_3)$ together with an $A$-module structure on $\huaL=L_0\oplus L_1$ and an $A$-module morphism $\theta:L_0\longrightarrow\Der(A)$, called the anchor, such that for all $a\in A$ and $X^0,Y^0\in L_0$ and $Y\in\huaL$, the following conditions are satisfied:
\begin{itemize}
\item[\rm(i)]$l_2(X^0,aY)=al_2(X^0,Y)+\theta(X^0)(a)Y;$
\item[\rm(ii)]for $i\neq2$, $l_i$ are $A$-linear;
\item[\rm(iii)]$\theta(l_2(X^0,Y^0))=[\theta(X^0),\theta(Y^0)]^c$;
\item[\rm(iv)]$\theta\circ l_1=0$.
\end{itemize}
We denote a Lie-Rinehart $2$-algebra by $(L_0,L_1,l_1,l_2,l_3,\theta)$ if $A$ is fixed, or simply by $\huaL$. A Lie-Rinehart $\K$-$2$-algebra $(L_0,L_1,l_1,l_2,l_3,\theta)$ is said to be \bf{skeletal} {\bf (strict)} if $l_1=0$ ($l_3=0$).
\end{defi}

In the following, we first recall the definition of pre-Lie $\K$-$2$-algebra.
\begin{defi}\label{defi:2pre-Lie}{\rm(\cite{Sheng19})}
  A   {\bf pre-Lie $\K$-$2$-algebra} $\huaV$ over $\K$ consists of the following data:
\begin{itemize}
\item[$\bullet$] a complex of vector spaces over $\K:V_{1}\stackrel{m_1}{\longrightarrow}V_0,$

\item[$\bullet$] bilinear maps $l_2:V_{i}\otimes V_{j}\longrightarrow
V_{i+j}$, where  $0\leq i+j\leq1$,

\item[$\bullet$] a  trilinear map $l_3:\wedge^2V_0\otimes V_0\longrightarrow
V_{1}$,
   \end{itemize}
   such that for all $w,x,y,z\in V_0$ and $m,n\in V_{1}$, the following equalities are satisfied:
\begin{itemize}
\item[$\rm(a)$] $m_1 m_2(x,m)=m_2(x,m_1(m)),$
\item[$\rm(b)$]$m_1 m_2(m,x)=m_2(m_1( m),x),$
\item[$\rm(c)$]$m_2(m_1(m),n)=m_2(m,m_1(n)),$
\item[$\rm(e_1)$]$m_1 m_3(x,y,z)=m_2(x,m_2(y,z))-m_2(m_2(x,y),z)-m_2(y,m_2(x,z))+m_2(m_2(y,x),z),$
\item[$\rm(e_2)$]$ m_3(x,y,m_1(m))=m_2(x,m_2(y,m))-m_2(m_2(x,y),m)-m_2(y,m_2(x,m))+m_2(m_2(y,x),m),$
\item[$\rm(e_3)$]$ m_3(m_1(m),x,y)=m_2(m,m_2(x,y))-m_2(m_2(m,x),y)-m_2(x,m_2(m,y))+m_2(m_2(x,m),y),$
\item[$\rm(f)$] \begin{eqnarray*}
&&m_2(w,m_3(x,y,z))-m_2(x,m_3(w,y,z))+m_2(y,m_3(w,x,z))+m_2(m_3(x,y,w),z)\\
&&-m_2(m_3(w,y,x),z)+m_2(m_3(w,x,y),z)-m_3(x,y,m_2(w,z))+m_3(w,y,m_2(x,z))\\
&&-m_3(w,x,m_2(y,z))-m_3(m_2(w,x)-m_2(x,w),y,z)+m_3(m_2(w,y)-m_2(y,w),x,z)\\
&&-m_3(m_2(x,y)-m_2(y,x),w,z)=0.\end{eqnarray*}
   \end{itemize}
\end{defi}
We usually denote a pre-Lie $\K$-$2$-algebra by
$(V_{1},V_0,m_1,m_2,m_3)$, or simply by
$\huaV$.

\begin{defi}\label{defi:PLR 2-A}
A {\bf pre-Lie-Rinehart $2$-algebra} over $A$ is a pair that consists of a pre-Lie $2$-$\K$-algebra $(P_0,P_1,m_1,m_2,m_3)$ together with an $A$-module structure on $\huaP=P_0\oplus P_1$ and an $A$-module morphism $\theta:P_0\longrightarrow\Der(A)$, called the anchor, such that for all $a\in A$ and $X^0,Y^0\in P_0$ and $Y\in\huaP$, the following conditions are satisfied:
\begin{itemize}
\item[\rm(i)]$m_2(X^0,aY)=am_2(X^0,Y)+\theta(X^0)(a)Y,\quad m_2(aX^0,Y)=am_2(X^0,Y)$;
\item[\rm(ii)]$m_2(X^1,aY)=m_2(aX^1,Y)=am_2(X^1,Y)$;
\item[\rm(iii)]for $i\neq2$, $m_i$ are $A$-linear;
\item[\rm(iv)]$\theta(m_2(X^0,Y^0)-m_2(Y^0,X^0))=[\theta(X^0),\theta(Y^0)]^c$;
\item[\rm(v)]$\theta\circ m_1=0$.
\end{itemize}
We denote a pre-Lie-Rinehart $2$-algebra by $(P_0,P_1,m_1,m_2,m_3,\theta)$ if $A$ is fixed, or simply by $\huaP$. A pre-Lie-Rinehart $\K$-$2$-algebra $(P_0,P_1,m_1,m_2,m_3,\theta)$ is said to be \bf{skeletal} {\bf (strict)} if $m_1=0$ ($m_3=0$).
\end{defi}

Given a pre-Lie-Rinehart $2$-algebra $(P_0,P_1,m_1,m_2,m_3,\theta)$, we define $l_2:P_i\wedge P_j\longrightarrow P_{i+j}$ and $l_3:\wedge^3 P_0\longrightarrow P_1$ by
\begin{eqnarray}
  \label{eq:LR2Al21}l_2(X^0,Y^0)&=&m_2(X^0,Y^0)-m_2(Y^0,X^0),\\
 \label{eq:LR2Al22} l_2(X^0,Y^1)&=&-l_2(Y^1,X^0)=m_2(X^0,Y^1)-m_2(Y^1,X^0),\\
  \label{eq:LR2Al3}l_3(X^0,Y^0,Z^0)&=&m_3(X^0,Y^0,Z^0)+m_3(Z^0,X^0,Y^0)+m_3(Y^0,Z^0,X^0),
\end{eqnarray}
where $X^0,Y^0,Z^0\in P_0$ and $Y^1\in P_1$.
\begin{pro}
  Let $\huaP=(P_0,P_1,m_1,m_2,m_3,\theta)$ be a pre-Lie-Rinehart $2$-algebra. Then $(P_0,P_1,l_1=m_1,l_2,l_3,\theta)$ is a Lie-Rinehart $2$-algebra, which we denote by $\huaP^c$, where $l_2$ and $l_3$ are given by \eqref{eq:LR2Al21}-\eqref{eq:LR2Al3} respectively.
\end{pro}
\begin{proof}
   It is straightforward to check that $(P_0,P_1,l_1=m_1,l_2,l_3)$ is a Lie $\K$-$2$-algebra. For $X^0,Y^0\in P_0$, $Y^1\in P_1$ and $a\in A$, we have
\begin{eqnarray*}
  l_2(X^0,aY^0)&=&m_2(X^0,aY^0)-m_2(aY^0,X^0)=am_2(X^0,Y^0)+\theta(X^0)(a)Y^0-am_2(Y^0,X^0)\\
  &=&a l_2(X^0,Y^0) +\theta(X^0)(a)Y^0
\end{eqnarray*}
and
\begin{eqnarray*}
  l_2(X^0,aY^1)&=&m_2(X^0,aY^1)-m_2(aY^1,X^0)=am_2(X^0,Y^1)+\theta(X^0)(a)Y-am_2(Y^1,X^0)\\
  &=&a l_2(X^0,Y^1) +\theta(X^0)(a)Y^1,
\end{eqnarray*}
which implies (i) in the definition of Lie-Rinehart $2$-algebra. The rest is obvious.
 \end{proof}

Let $\huaP=(P_0,P_1,m_1,m_2,m_3,\theta)$ be a skeletal pre-Lie-Rinehart $2$-algebra. Condition $\rm(e_1)$ in Definition \ref{defi:2pre-Lie} and (i) and (iii) in Definition \ref{defi:PLR 2-A} imply that $(P_0,m_2,\theta)$ is a pre-Lie-Rinehart algebra. Conditions $\rm(e_2)$ and $\rm(e_3)$ in Definition \ref{defi:2pre-Lie} and (i) and (ii) in Definition \ref{defi:PLR 2-A} imply that $\rho$ and $\mu$ given by
\begin{eqnarray*}
  \rho(X^0)(Y^1)=m_2(X^0,Y^1),\quad \mu(X^0)(Y^1)=m_2(Y^1,X^0),\quad\forall~X^0\in P_0, Y^1\in P_1,
\end{eqnarray*}
give a representation of pre-Lie-Rinehart algebra  $(P_0,m_2,\theta)$ on $P_1$. Furthermore, Condition $\rm(f)$ in Definition \ref{defi:2pre-Lie} means that $m_3$ is a $3$-cocycle on $P_0$ with values in $P_1$. Summarize the discussion above, we have
\begin{pro}
  There is a one-to-one corresponding between skeletal pre-Lie-Rinehart $2$-algebras and triples $((P_0,m_2,\theta),(P_1;\rho,\mu),m_3)$, where $(P_0,m_2,\theta)$ is a pre-Lie-Rinehart algebra, $(P_1;\rho,\mu)$ is a representation of $(P_0,m_2,\theta)$ and $m_3$ is a $3$-cocycle on $P_0$ with values in $P_1$.
\end{pro}

\begin{thm}\label{one to one correspondence}
  There is a one-to-one correspondence between strict (pre-)Lie-Rinehart $2$-algebras and crossed modules for (pre-)Lie-Rinehart algebras.
\end{thm}
\begin{proof}
We only prove this conclusion for strict pre-Lie-Rinehart $2$-algebras and crossed modules for pre-Lie-Rinehart algebras. The other case can be proved similarly. Let $\huaP=(P_0,P_1,m_1,m_2,m_3,\theta)$ be a strict pre-Lie-Rinehart $2$-algebra. We construct a crossed module for pre-Lie-Rinehart algebras as follows. Obviously, $(P_0,\cdot_{P_0}=m_2,\theta)$ is a pre-Lie-Rinehart algebra. Define a multiplication $\cdot_{P_1}:P_1\times P_1\longrightarrow P_1$ by
\begin{equation}
  X^1\cdot_{P_1} Y^1=m_2(m_1(X^1),Y^1)=m_2(X^1,m_1(Y^1)),\quad\forall~X^1,Y^1\in P_1.
\end{equation}
Then by Condition (a) and $\rm(e_2)$ in Definition \ref{defi:2pre-Lie}, $(P_1,\cdot_{P_1})$ is a pre-Lie $\K$-algebra. By Condition (a) in Definition \ref{defi:2pre-Lie} and $m_1$ is $A$-linear, we deduce that $m_1$ is an $A$-linear homomorphism between pre-Lie algebras. Conditions $\rm(e_2)$ and $\rm(e_3)$ in Definition \ref{defi:2pre-Lie} and (i) and (ii) in Definition \ref{defi:PLR 2-A} implies that $\rho$ and $\mu$ given by
\begin{eqnarray*}
  \rho(X^0)(Y^1)=m_2(X^0,Y^1),\quad \mu(X^0)(Y^1)=m_2(Y^1,X^0),\quad\forall~X^0\in P_0, Y^1\in P_1,
\end{eqnarray*}
give a representation of pre-Lie-Rinehart algebra  $(P_0,\cdot_{P_0},\theta)$ on $P_1$. By Conditions (a) and $(b)$, we deduce that (1) in Definition \ref{defi:cmodPLRA} holds. (2) in Definition \ref{defi:cmodPLRA} follows from the definition of $\cdot_{P_1}$ directly. (3) in Definition \ref{defi:cmodPLRA} follows from (v) in Definition \ref{defi:PLR 2-A}. Thus $((P_0,\cdot_{P_0}),(P_1,\cdot_{P_1}),m_1,(\rho,\mu))$ constructed above is a crossed module for pre-Lie-Rinehart algebras.

Conversely, a crossed module for pre-Lie-Rinehart algebras $((P_0,\cdot_{P_0}),(P_1,\cdot_{P_1}),\partial,(\rho,\mu))$ gives rises to a strict pre-Lie-Rinehart $2$-algebra $(P_0,P_1,m_1=\partial,m_2,m_3=0)$, where $m_2:P_i\otimes P_j\longrightarrow P_{i+j},\quad 0\leq i+j\leq1$ is given by
$$m_2(X^0,Y^0)=X^0\cdot_{P_0}Y^0,\quad m_2(X^0,Y^1)=\rho(X^0)(Y^1)\quad m_2(Y^1,X^0)=\mu(X^0)(Y^1).$$
The crossed module conditions give various conditions for a strict pre-Lie-Rinehart $2$-algebra. We omit the details.
\end{proof}

\noindent
{\bf Acknowledgements. } This research is supported by NSFC (11901501). We give our warmest thanks to Chengming Bai and Yufeng Pei for very useful comments and discussions.

\vspace{3mm}
\noindent

 \end{document}